\documentclass[12pt]{article}
\usepackage {epsfig,amsmath,euscript}
\usepackage[latin1]{inputenc}
\usepackage[english]{babel}
\usepackage{color}
\usepackage{amsmath}
\usepackage{amssymb}
\usepackage[scr=boondoxo]{mathalfa}
\usepackage[amsmath,thmmarks,standard,thref]{ntheorem}
\usepackage{ae}
\usepackage{subfigure}
\usepackage[T1]{fontenc}
\usepackage{graphicx}
\usepackage{epstopdf}
\usepackage{soul}
\usepackage{color}
\sethlcolor{yellow}
\definecolor{rred}{rgb}{0.7,0.0,0.2}
\definecolor{bblue}{rgb}{0.2,0.0,0.7}

\newcommand{\secref}[1]{\ref{sec:#1}}

\newcommand{\seclab}[1]{\label{sec:#1}}
\newcommand{\eqlab}[1]{\label{eq:#1}}

\renewcommand{\eqref}[1]{(\ref{eq:#1})}
\newcommand{\eqsref}[2]{(\ref{eq:#1}) and~(\ref{eq:#2})}

\newcommand{\figref}[1]{Fig.~\ref{fig:#1}}
\newcommand{\figlab}[1]{\label{fig:#1}}
\newcommand{\propref}[1]{Proposition~\ref{proposition:#1}}
\newcommand{\proplab}[1]{\label{proposition:#1}}
\newcommand{\lemmaref}[1]{Lemma~\ref{lemma:#1}}
\newcommand{\lemmalab}[1]{\label{lemma:#1}}
\newcommand{\remref}[1]{Remark~\ref{remark:#1}}
\newcommand{\remlab}[1]{\label{remark:#1}}
\newcommand{\thmref}[1]{Theorem~\ref{theorem:#1}}
\newcommand{\thmlab}[1]{\label{theorem:#1}}

\usepackage{tikz}

\title{On the regularization of impact without collision: the Painlev\'e paradox and compliance}

\author{S. J. Hogan and K. Uldall Kristiansen\thanks{S. J. Hogan: Department of Engineering Mathematics, University of Bristol, Bristol BS8 1UB, United Kingdom. K. Uldall Kristiansen: Department of Applied Mathematics and Computer Science, Technical University of Denmark, 2800 Kgs. Lyngby, DK. }}
\begin{document}
\maketitle

\begin{abstract}
We consider the problem of a rigid body, subject to a unilateral constraint, in the presence of Coulomb friction. We regularize the problem by assuming compliance (with both stiffness and damping) at the point of contact, for a general class of normal reaction forces. Using a rigorous mathematical approach, we recover impact without collision (IWC)  in both the inconsistent and indeterminate Painlev\'e paradoxes, in the latter case giving an exact formula for conditions that separate IWC and lift-off. We solve the problem for arbitrary values of the compliance damping and give explicit asymptotic expressions in the limiting cases of small and large damping, all for a large class of rigid bodies. 
\end{abstract}


\pagestyle{myheadings}
\thispagestyle{plain}
\section{Introduction}\seclab{intro}
In mechanics, in problems with unilateral constraints in the presence of friction, the rigid body assumption can result in the governing equations having multiple solutions (the {\it indeterminate} case) or no solutions (the {\it inconsistent} case). The classical example of Painlev\'e \cite{Painleve1895, Painleve1905a,Painleve1905b}, consisting of a slender rod slipping\footnote{We prefer to avoid describing this phase of the motion as {\it sliding} because we will be using ideas from piecewise smooth systems \cite{filippov1988differential}, where sliding has exactly the opposite meaning.} along a rough surface (see  \figref{fig:rod}), is the simplest and most studied example of these phenomena, now known collectively as {\it Painlev\'e paradoxes} \cite{BlumenthalsBrogliatoBertails2016, Brogliato1999, ChampneysVarkonyi2016, ShenStronge2011, Stewart2000}. Such paradoxes can occur at physically realistic parameter values in many important engineering systems \cite{LeineBrogliatoNijmeijer2002,LiuZhaoChen2007,NeimarkFufayev1995, Or2014, OrRimon2012, WilmsCohen1981, ZhaoLiuMaChen2008}.

When a system has no {\it consistent} solution, it can not remain in that state. Lecornu \cite{Lecornu1905} proposed a jump in vertical velocity to escape an inconsistent, horizontal velocity, state. This jump has been called {\it impact without collision} (IWC) \cite{GenotBrogliato1999}, {\it tangential impact} \cite{Ivanov1986} or {\it dynamic jamming} \cite{OrRimon2012}. Experimental evidence of IWC is given in \cite{ZhaoLiuMaChen2008}. 

IWC occurs instantaneously. So it must be incorporated into the rigid body formulation \cite{Darboux1880, Keller1986} by considering the equations of motion in terms of the normal impulse, rather than time. However, this process has been controversial \cite{Brach1997, Stronge2015}, because it can sometimes lead to an apparent energy gain in the presence of friction. 

G\'enot and Brogliato \cite{GenotBrogliato1999} considered the dynamics around a critical point, corresponding to zero vertical acceleration of the end of the rod. They proved that, when starting in a consistent state, the rod must stop slipping before reaching the critical point. In particular, paradoxical situations cannot be reached after a period of slipping. 

One way to address the Painlev\'e paradox is to {\it regularize} the rigid body formalism. Physically this often corresponds to assuming some sort of compliance at the contact point $A$, typically thought of as a spring, with stiffness (and sometimes damping) that tend to the rigid body model in a suitable limit. Mathematically, very little rigorous work has been done on how IWC and Painlev\'e paradoxes can be regularized. Dupont and Yamajako \cite{DupontYamajako1997} treated the problem as a slow fast system, as we will do. They explored the fast time scale dynamics, which is unstable for the Painlev\'e paradoxes. Song {\it et al.} \cite{SongKrausKumarDupont2001} established conditions under which these dynamics can be stabilized. Le Suan An \cite{Lesuanan1990} considered a system with bilateral constraints and showed qualitatively the presence of a regularized IWC as a jump in vertical velocity from a compliance model with diverging stiffness. Zhao {\it et al.} \cite{ZhaoLiuChenBrogliato2015} considered the example in \figref{fig:rod} and regularized the equations by assuming a compliance that consisted of an {\it undamped} spring. They estimated, as a function of the stiffness, the orders of magnitude of the time taken in each phase of the (regularized) IWC.  Another type of regularization was considered by Neimark and Smirnova \cite{NeimarkSmirnova2001} who assumed that the normal and tangential reactions took (different) finite times to adjust. 

In this paper, we present the first rigorous analysis of the regularized rigid body formalism, in the presence of compliance with both stiffness {\it and} damping. We recover impact without collision (IWC)  in both the inconsistent and indeterminate cases and, in the latter case, we present a formula for conditions that separate IWC and lift-off.  We solve the problem for arbitrary values of the compliance damping and give explicit asymptotic expressions in the limiting cases of small and large damping. Our results apply directly to a general class of rigid bodies. Our approach is similar to that used in \cite{krihog,krihog2} to understand the forward problem in piecewise smooth (PWS) systems in the presence of a two-fold. 

The paper is organized as follows. In Section \secref{classic}, we introduce the problem, outline some of the main results known to date and include compliance. In Section \secref{mainResults}, we give a summary of our main results, \thmref{thm:main} and \thmref{cor}, before presenting their derivation in Sections \secref{blowup} and \secref{new}. We discuss our results in Section \secref{discussion} and outline our conclusions in Section \secref{conclusions}. 
\section{Classical Painlev\'e problem}\seclab{classic}
Consider a rigid rod $AB$, slipping on a rough horizontal surface, as depicted in \figref{fig:rod}. 

\begin{figure}[h!] 
\begin{center}
{\includegraphics[width=.5\textwidth]{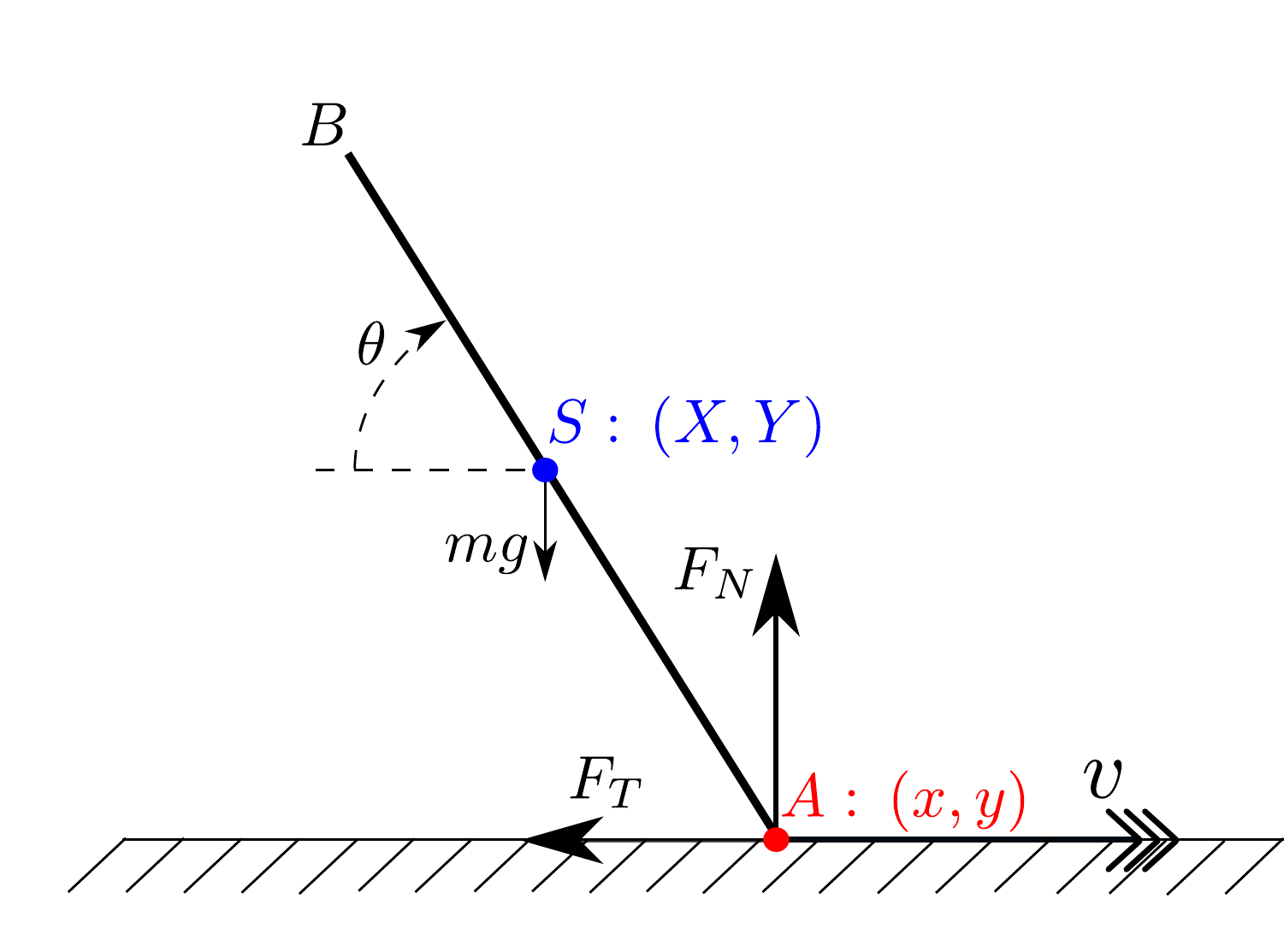}}
\end{center}
 \caption{
 The classical Painlev\'e problem.}
 \figlab{fig:rod}
\end{figure}
The rod has mass $m$, length $2l$, the moment of inertia of the rod about its center of mass $S$ is given by $I$ and its center of mass coincides with its center of gravity. The point $S$ has coordinates $(X,Y)$ relative to an inertial frame of reference $(x,y)$ fixed in the rough surface. The rod makes an angle $\theta$ with respect to the horizontal, with $\theta$ increasing in a clockwise direction. At $A$, the rod experiences a contact force $(F_T, F_N)$, which opposes the motion. The dynamics of the rod is then governed by the following equations
\begin{eqnarray}\eqlab{eq:dynrod}
m\ddot{X} & = & -F_T, \\ 
m\ddot{Y} & = & -mg - F_N, \nonumber \\
I\ddot{\theta} & = & -l(\cos\theta F_N-\sin\theta F_T). \nonumber
\end{eqnarray}
where $g$ is the acceleration due to gravity.

The coordinates $(X,Y)$ and $(x,y)$ are related geometrically as follows
\begin{eqnarray}\eqlab{eq:coord}
x & = & X + l\cos\theta, \\
y & = & Y - l\sin\theta. \nonumber
\end{eqnarray}

We now adopt the scalings $(X,Y)=l(\tilde X,\tilde Y), \, (x,y)=l(\tilde x,\tilde y), \, (F_T, F_N)=mg(\tilde F_T, \tilde F_N), \, t=\frac{1}{\omega}\tilde t, \, \alpha=\frac{ml^2}{I}$ where $\omega^2=\frac{g}{l}$. For a uniform rod, $I=\frac{1}{3}ml^2$, and so $\alpha=3$ in this case. 


Then for general $\alpha$, \eqsref{eq:dynrod}{eq:coord} can be combined to become, on dropping the tildes,
\begin{eqnarray}\eqlab{eq:dynrodsc}
\ddot{x} & = & -\dot{\theta}^2\cos \theta +\alpha\sin\theta\cos\theta F_N-(1+\alpha\sin^2\theta)F_T, \\ 
\ddot{y} & = & -1 +\dot{\theta}^2\sin\theta + (1+\alpha\cos^2\theta) F_N -\alpha\sin\theta\cos\theta F_T,  \nonumber \\
\ddot{\theta} & = & -\alpha(\cos\theta F_N-\sin\theta F_T). \nonumber
\end{eqnarray}
To proceed, we need to determine the relationship between $F_N$ and $F_T$. We assume Coulomb friction between the rod and the surface. Hence, when $\dot{x} \ne 0$, we set
\begin{equation}\eqlab{eq:Coulomb}
F_T = \mu \textnormal{sign}(\dot{x})F_N,
\end{equation}
where $\mu$ is the coefficient of friction. By substituting \eqref{eq:Coulomb} into \eqref{eq:dynrodsc}, we obtain two sets of governing equations for the motion, depending on the sign of $\dot{x}$, as follows:
\begin{eqnarray}\eqlab{eq:Painleve}
\dot{x} & = & v, \\
\dot{v} & = & a(\theta,\phi) + q_{\pm}(\theta)F_N, \nonumber \\
\dot{y} & = & w, \nonumber \\
\dot{w} & = & b(\theta,\phi) + p_{\pm}(\theta)F_N, \nonumber \\
\dot{\theta} & = & \phi, \nonumber \\
\dot{\phi} & = & c_{\pm}(\theta)F_N, \nonumber 
\end{eqnarray}
where the variables $v,w,\phi$ denote velocities in the $x,y,\theta$ directions respectively and
\begin{eqnarray}\eqlab{eq:coeffs}
a(\theta,\phi) & = & -\phi^2\cos \theta , \\
b(\theta,\phi) & = & -1 +\phi^2\sin\theta, \nonumber \\
q_{\pm}(\theta) & = & \alpha\sin\theta\cos\theta \mp \mu(1+\alpha\sin^2\theta), \nonumber \\
p_{\pm}(\theta) & = & 1+\alpha\cos^2\theta \mp \mu\alpha\sin\theta\cos\theta, \nonumber \\
c_{\pm}(\theta) & = & -\alpha(\cos\theta \mp \mu\sin\theta) \nonumber 
\end{eqnarray}
for the configuration in \figref{fig:rod}. The suffices $q_{\pm},p_{\pm},c_{\pm}$ correspond to $\dot{x} = v \gtrless 0$ respectively. 

System \eqref{eq:Painleve} is a Filippov system \cite{filippov1988differential}. Hence we obtain a well-defined forward flow when $\dot{x} = v = 0$ and
\begin{align}
a(\theta,\phi) + q_{+}(\theta)F_N<0<a(\theta,\phi) + q_{-}(\theta)F_N,\eqlab{filineq} 
\end{align}
where $\dot v$ in \eqref{eq:Painleve}$_{\pm}$ for $v\gtrless 0$ both oppose $v=0$, by using the Filippov vector-field \cite{filippov1988differential}. Simple computations give:
\begin{proposition}\proplab{lem:Filippov}
The Filippov vector-field, within the subset of the switching manifold $\dot{x} =v=0$ where \eqref{filineq} holds, is given by
 \begin{align}
 \dot{y}&=w,\eqlab{eq:Filippov}\\
 \dot w &= b(\theta,\phi)+S_w(\theta) F_N,\nonumber\\
 \dot \theta &=\phi,\nonumber\\
 \dot \phi &=S_\phi(\theta)F_N,\nonumber
 \end{align}
 where
 \begin{align}\eqlab{eq:Sw} 
  S_w(\theta) &= \frac{q_-(\theta) }{q_-(\theta)-q_+(\theta)}p_+(\theta)-\frac{q_+(\theta) }{q_-(\theta)-q_+(\theta)}p_-(\theta)=\frac{1+\alpha}{1+\alpha \sin^2 \theta},\\
  S_\phi(\theta) &= \frac{q_-(\theta) }{q_-(\theta)-q_+(\theta)}c_+(\theta)-\frac{q_+(\theta) }{q_-(\theta)-q_+(\theta)}c_-(\theta)\nonumber=-\frac{\alpha\cos\theta}{1+\alpha\sin^2\theta}\nonumber.
 \end{align}
\end{proposition}

\begin{remark}
Our results hold for mechanical systems with different $q_\pm$, $p_\pm$ and $c_\pm$  in \eqref{eq:coeffs} and even dependency on several angles $\theta\in \mathbb T^d$, e.g. the two-link mechanism of Zhao {\it et al.} \cite{ZhaoLiuMaChen2008}. Note that both $S_w$ and $S_{\phi}$ in \eqref{eq:Sw} are independent of $\mu$, even for general $c_\pm$, $q_\pm$ and $p_\pm$. 
\end{remark}


In order to solve \eqsref{eq:Painleve}{eq:Filippov}, we need to determine $F_N$.  The constraint-based method leads to the Painlev\'e paradox. The compliance-based method is the subject of this paper. 

\subsection{Constraint-based method}
In order that the constraint $y=0$ be maintained, $\ddot{y} \enskip ( = \dot{w})$ and $F_N$ form a complementarity pair given by
\begin{equation}\eqlab{eq:comppair}
\dot{w} \ge 0, \quad F_N \ge 0, \quad F_N \cdot \dot{w} =0.
\end{equation}
Note that $F_N\ge 0$ since the rough surface can only push, not pull, the rod. Then for general motion of the rod, $F_N$ and $y$ satisfy the complementarity conditions 
\begin{equation}\eqlab{eq:comprel}
0 \le F_N \perp y \ge 0.
\end{equation}
In other words, at most one of $F_N$ and $y$ can be positive. 

For the system shown in \figref{fig:rod}, the Painlev\'e paradox occurs when $v>0$ and $\theta \in (0,\frac{\pi}{2})$, provided $p_+(\theta)<0$, as follows: From the fourth equation in \eqref{eq:Painleve}, we can see that $b$ is the free acceleration of the end of the rod. Therefore if $b>0$, lift-off is always possible when $y=0,\,w=0$. But if $b<0$, in equilibrium we would expect a forcing term $F_N$ to maintain the rod on $y=0$. From $\dot w=0$ we obtain
 \begin{align}
  F_N = -\frac{b}{p_+}\eqlab{eq:kk}
 \end{align}
since $v>0$. If $p_+>0$, which is always true for $\theta \in (\frac{\pi}{2},\pi)$, then $F_N\ge 0$, in line with \eqref{eq:comprel}. But if $p_+<0$, which can happen if $\theta \in (0,\frac{\pi}{2})$, then $F_N<0$ in \eqref{eq:kk}. Then $F_N$ is in an {\it inconsistent} (or {\it non-existent}) mode. On the other hand, if $b>0$ and $p_+<0$ then $F_N>0$ in \eqref{eq:kk}. At the same time lift-off is also possible from $y=0$ and hence $F_N$ is in an {\it indeterminate} (or {\it non-unique}) mode. It is straightforward to show that $p_+(\theta)<0$ requires
\begin{equation}\eqlab{eg:mucrit}
\mu > \mu_P (\alpha) \equiv \frac{2}{\alpha}\sqrt{1+\alpha}.
\end{equation}
Then the Painlev\'e paradox can occur for $\theta \in (\theta_{1},\theta_{2})$ where
\begin{eqnarray}\eqlab{eq:thetacrit}
\theta_{1}(\mu,\alpha) & = & \arctan \frac{1}{2}\left ( \mu\alpha - \sqrt{\mu^2\alpha^2-4(1+\alpha)} \right ), \\
\theta_{2}(\mu,\alpha) & = & \arctan \frac{1}{2}\left ( \mu\alpha + \sqrt{\mu^2\alpha^2-4(1+\alpha)} \right ). \nonumber
\end{eqnarray}

For a uniform rod with $\alpha=3$, we have $\mu_P(3) = \frac{4}{3}$. For $\alpha=3$ and $\mu=1.4$ the dynamics can be summarized\footnote{Compare with Figure 2 of G\'enot and Brogliato \cite{GenotBrogliato1999}, where the authors plot the {\it unscaled} angular velocity $\omega\phi$ {\it vs.} $\theta$, for the case $g=9.8$ $\textnormal{ms}^{-2}$, $l=1$ m.} in the $(\theta, \phi)$-plane, as in \figref{fig:GB}. 
\begin{figure}[h!] 
\begin{center}
{\includegraphics[width=.7\textwidth]{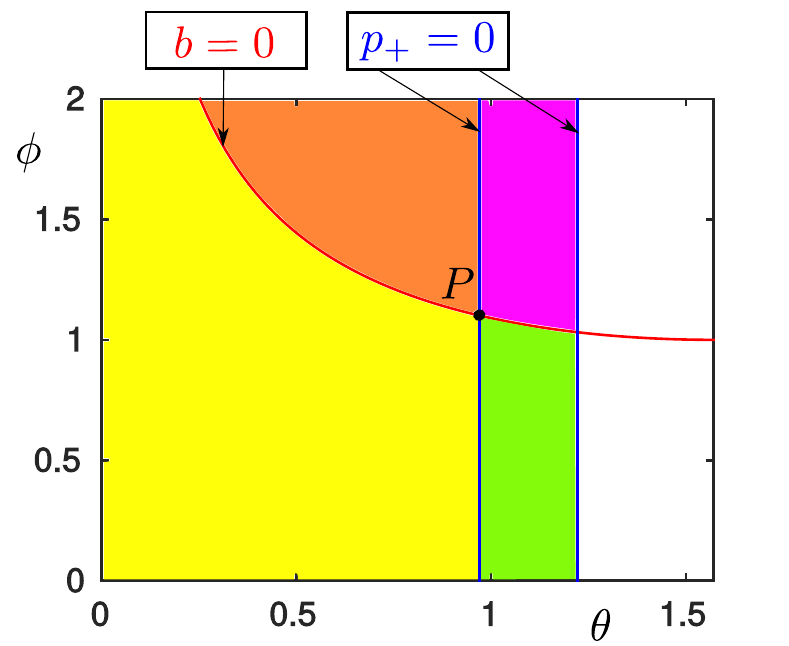}}
\end{center}
 \caption{The $(\theta, \phi)$-plane for the classical Painlev\'e problem of \figref{fig:rod}, for $\alpha=3$ and $\mu=1.4$. The point $P$ has coordinates $(\theta_1, \sqrt{\csc\theta_1})$, where $\theta_1$ is given in \eqref{eq:thetacrit}. In the first quadrant centered on $P$, we have $b >0, \enskip p_+<0$, so the dynamics is indeterminate (non-unique). In the second quadrant, $b >0, \enskip p_+>0$ and the rod lifts off the rough surface. In the third quadrant, $b <0, \enskip p_+>0$ and the rod moves (slips) along the surface. Here G\'enot and Brogliato \cite{GenotBrogliato1999} showed that the dynamics cannot cross $p_+=0$ unless also $b=0$. In the fourth quadrant, $b <0, \enskip p_+<0$ and the dynamics is inconsistent (non-existent). Even though the constraint $y=0$ is satisfied, there exists no positive value of $F_N$, contradicting \eqref{eq:comprel}.}
 \figlab{fig:GB}
\end{figure}
Along $\theta=\theta_{1},\theta_{2}$, we have $p_+(\theta)=0$. These lines intersect the curve $b(\theta,\phi)=0$ at four points: $\phi^{\pm}_{1,2} = \pm \sqrt{\csc\theta_{1,2}}$. G\'enot and Brogliato \cite{GenotBrogliato1999} showed that the point $P: (\theta,\phi) = (\theta_1, \sqrt{\csc\theta_1})$ is the most important and analyzed the local dynamics around it. The rigid body equations \eqref{eq:dynrod} are unable to resolve the dynamics in the third and fourth quadrants. So we regularize these equations using compliance. 

\subsection{Compliance-based method}\seclab{sec:compliancebasedmethod}
We assume that there is compliance at the point $A$ between the rod and the surface, when they are in contact (see \figref{fig:rod}). Following \cite{DupontYamajako1997,McClamroch1989}, we assume that there are small excursions into $y<0$. Then we require that the nonnegative normal force $F_N(y,w)$ is a PWS function of $(y,w)$:  
\begin{eqnarray}\eqlab{eq:compliance}
F_N(y,w)  =\left[f(y,w)\right]\equiv \left\{ \begin{array}{cc}
                    0 & \text{for} \quad  y>0\\
                    \max \{f(y,w), 0 \} & \text{for} \quad y \le 0,
                    \end{array}\right.
\end{eqnarray}
where the operation $\left[\cdot\right]$ is defined by the last equality and $f(y,w)$ is assumed to be a smooth function of $(y,w)$ satisfying
$\partial_y f<0,  \partial_w f<0.$ The quantities $\partial_y f(0,0)$ and $\partial_w f(0,0)$ represent a (scaled) spring constant and damping coefficient, respectively.
We are interested in the case when the compliance is very large, so we introduce a small parameter $\epsilon$ as follows:
\begin{equation}\eqlab{eq:epsilon}
\partial_y f(0,0)=\epsilon^{-2}, \quad  \partial_w f(0,0) = \epsilon^{-1}\delta.
\end{equation}
This choice of scaling \cite{DupontYamajako1997, McClamroch1989} ensures that the critical damping coefficient ($\delta_{crit}=2$ in the classical Painlev\'e problem) is independent of $\epsilon$. Our analysis can handle any $f$ of the form
$f(y,w) = \epsilon^{-1} h(\epsilon^{-1} y,w)$ with
\begin{align}
 h(\hat y,w) = -\hat y-\delta w+\mathcal O((\hat y+w)^2).\eqlab{hNon}
\end{align}
 But, to obtain our quantitative results, we truncate \eqref{hNon} and consider the linear function
 \begin{align}\eqlab{hLin}
 h(\hat y,w) = -\hat y-\delta w,
 \end{align}
 so that
 \begin{align}\eqlab{eq:FN}
  F_N(y,w) = \epsilon^{-1}\left[-\epsilon^{-1} y-\delta w\right].
 \end{align}

In what follows, the first equation in \eqref{eq:Painleve} will play no role, so we drop it from now on. Then we combine the remaining five equations in \eqref{eq:Painleve} with \eqsref{eq:compliance}{eq:epsilon} to give the following set of governing equations that we will use in the sequel 
\begin{eqnarray}\eqlab{eq:PainleveEqs}
\dot{y} & = & w, \\
\dot{w} & = & b(\theta,\phi) + p_{\pm}(\theta)\epsilon^{-1}[-\epsilon^{-1}y-\delta w], \nonumber \\
\dot{\theta} & = & \phi, \nonumber \\
\dot{\phi} & = & c_{\pm}(\theta)\epsilon^{-1}[-\epsilon^{-1}y-\delta w], \nonumber \\
\dot{v} & = & a(\theta,\phi) + q_{\pm}(\theta)\epsilon^{-1}[-\epsilon^{-1}y-\delta w], \nonumber
\end{eqnarray}

For $\epsilon>0$ this is a well-defined Filippov system.
The slipping region \eqref{filineq} and the Filippov vector-field \eqref{eq:Filippov} are obtained by replacing $F_N$ in these expressions with the square bracket $\epsilon^{-1}[-\epsilon^{-1} y-\delta w]$ (see also \lemmaref{lemma:nonStick} below). 

\section{Main Results}\seclab{mainResults}
We now present the main results of our paper, \thmref{thm:main} and \thmref{cor}. \thmref{thm:main} shows that, if the rod starts in the fourth quadrant of \figref{fig:GB}, it undergoes a (regularized) IWC for a time of $\mathcal O(\epsilon \ln \epsilon^{-1})$. The same theorem also gives expressions for the resulting vertical velocity of the rod in terms of the compliance damping and initial horizontal velocity and orientation of the rod.  
\begin{theorem}\thmlab{thm:main}
 Consider an initial condition 
 \begin{align}
  (y,w,\theta, \phi,v) = (0,\mathcal O(\epsilon),\theta_0,\phi_0,v_0),\quad v_0>0,\eqlab{eq:ics}
 \end{align}
within the region of inconsistency (non-existence) where
\begin{align}
 p_+(\theta_0)<0,\quad b(\theta_0,\phi_0)<0,\eqlab{eq:nonExistence}
\end{align}
and $q_+(\theta_0)<0$, $q_-(\theta_0)>0$, $a\ne 0$.
Then the forward flow of \eqref{eq:ics} under \eqref{eq:PainleveEqs} returns to $\{(y,w,\theta, \phi,v)\vert y=0\}$ after a time $\mathcal O(\epsilon \ln \epsilon^{-1})$ with
\begin{align}
   w &= e(\delta,\theta_0)v_0+o(1),\eqlab{expr}\\
   \theta &=\theta_0+o(1),\nonumber\\
    \phi &=\phi_0+\left\{-\frac{c_+(\theta_0)}{q_+(\theta_0)}+\frac{S_\phi(\theta_0)}{S_w(\theta_0)}\left(e(\delta,\theta_0)+\frac{p_+(\theta_0)}{q_+(\theta_0)}\right)\right\}v_0+o(1),\nonumber\\
    v&=o(1),\nonumber
  \end{align}
  as $\epsilon\rightarrow 0$. During this time $y=\mathcal O(\epsilon),\,w=\mathcal O(1)$ so that $F_N=\mathcal O(\epsilon^{-1})$. The function $e(\delta,\theta_0)$, given in \eqref{eq:eGeneral} below, is smooth and monotonic in $\delta$ and has the following asymptotic expansions:
 \begin{align}
  e(\delta,\theta_0 ) &=  \frac{p_-(\theta_0)-p_+(\theta_0)}{ q_-(\theta_0)p_+(\theta_0)-q_+(\theta_0) p_-(\theta_0)}\delta^{-2}\left(1+\mathcal O(\delta^{-2} \ln \delta^{-1})\right) \quad \text{for}\quad \delta \gg 1,\eqlab{eDeltaLarge}\\                              
  e(\delta,\theta_0 ) &=  \sqrt{\frac{p_+(\theta_0)(p_-(\theta_0)-p_+(\theta_0))}{q_+(\theta_0)(q_-(\theta_0)-q_+(\theta_0))}} \Bigg(1-\nonumber\\                               
  &\frac{\sqrt{S_w(\theta_0)}}{2}\left(\pi-\textnormal{arctan}\left(\sqrt{-\frac{S_w(\theta_0)}{p_+(\theta_0)}}\right) \right)\delta+\mathcal O(\delta^2)\Bigg) \quad \text{for}\quad \delta \ll 1.       \eqlab{eDeltaSmall}                        
 \end{align}

\end{theorem}

\thmref{cor} is similar to \thmref{thm:main}, but now the rod starts in the first quadrant of \figref{fig:GB}. This theorem also gives an exact formula for initial conditions that separate (regularized) IWC and lift off.
 \begin{theorem}\thmlab{cor}
Consider an initial condition
 \begin{align}
 (y,w,\theta, \phi,v) = (0,\epsilon w_{10},\theta_0,\phi_0,v_0),\quad w_{10}<w_{1*}\equiv -\lambda_-(\theta_0)\frac{b(\theta_0,\phi_0)}{p_+(\theta_0)}<0,\eqlab{eq:icsNew}
 \end{align}
 with $\lambda_-$ defined in \eqref{eq:lambdapm} below,
within the region of indeterminacy (non-uniqueness) where
\begin{align}
 p_+(\theta_0)<0,\quad b(\theta_0,\phi_0)>0,\eqlab{eq:nonUnique}
\end{align}
and $q_+(\theta_0)<0$, $q_-(\theta_0)>0$, $a\ne 0$. 
  Then the conclusions of \thmref{thm:main}, including expressions \eqref{expr}, \eqsref{eDeltaLarge}{eDeltaSmall}, still hold true as $\epsilon \rightarrow 0$. For $w_{10}>w_{1*}$ lift-off occurs directly after a time $\mathcal O(\epsilon)$ with $w=\mathcal O(\epsilon)$. During this period $y=\mathcal O(\epsilon^2)$, 
 so $F_N=\mathcal O(1)$. 
 \end{theorem}
%
%
%
%

\begin{remark}\remlab{rem:main}
These two theorems have not appeared before in the literature. In the rigid body limit ($\epsilon \rightarrow 0$), we recover IWC in both  cases. Previous authors have either not carried out the ``very difficult" calculation \cite{McClamroch1989}, performed numerical calculations \cite{ChampneysVarkonyi2016, DupontYamajako1997} or given a range of estimates for the time of (regularized) IWC in the {\it absence} of damping \cite{ZhaoLiuChenBrogliato2015}. We give exact and asymptotic expressions for key quantities as well as providing a geometric interpretion of our results, for a large class of rigid bodies, in the presence of a large class of normal forces, as well as giving a precise estimate for the time of (regularized) IWC, all in the presence of both stiffness and damping. 
Note that we are not attempting to describe all the dynamics around $P$. There is a canard connecting the third quadrant with the first and the analysis is exceedingly complicated \cite{3rd} due to fast oscillatory terms. Instead, we follow \cite{ZhaoLiuChenBrogliato2015} and consider that the rod dynamics starts in a configuration with $p_+(\theta_0)<0$. 
\end{remark}

\section{Proof of \thmref{thm:main}: IWC in the inconsistent case}\seclab{blowup}

The proof of \thmref{thm:main} is divided into three phases, illustrated in  \figref{fig:IWC}. These phases are a generalisation of the phases of IWC in its rigid body formulation \cite{ZhaoLiuMaChen2008}. 
\begin{itemize}
 \item Slipping compression (section \secref{sec:slipping}): During this phase $y$, $w$ and $v$ all decrease. The dynamics follow an unstable manifold $\gamma^u$ of a set of critical points $C$, given in \eqref{C} below, as $\epsilon \rightarrow 0$. Along $\gamma^u$ the normal force $F_N=\mathcal O(\epsilon^{-1})$ and $v$ will therefore quickly decrease to $0$. Mathematically this part is complicated by the fact that the initial condition \eqref{eq:ics} belongs to the critical set $C$ as $\epsilon\rightarrow 0$. 
   \item Sticking (section \secref{sec:stick}): Since $F_N=\mathcal O(\epsilon^{-1})$ and $q_+q_-<0$ the rod will stick with $v\equiv 0$. During this phase $\ddot y=\dot w>0$ and eventually sticking ends with $F_N=0$ as $\epsilon\rightarrow 0$. 
 \item Lift-off (section \secref{sec:lift}): In the final phase $F_N = 0$, lift off occurs and the system eventually returns to $y=0$. 
\end{itemize}


\subsection{Slow-fast setting: Initial scaling}\seclab{sec:scaling}
Before we consider the first phase of IWC, we apply the scaling 
\begin{align}\eqlab{eq:initialscaling}
y = \epsilon \hat y, 
\end{align}
also used in \cite{DupontYamajako1997, McClamroch1989}, which brings the two terms in \eqref{eq:FN}
to the same order. Now let
\begin{align}
 \hat F_N(\hat y,w) \equiv \epsilon F_N(\epsilon \hat y,w)=\left[-\hat y-\delta w\right].\eqlab{eq:hatFNHaty}
\end{align}
Equations \eqref{eq:PainleveEqs} then read:
\begin{align}
 {\hat y}' &= w,\eqlab{eq:haty}\\
 w' &=\epsilon b(\theta,\phi)+p_\pm (\theta) \hat F_N(\hat y,w),\nonumber\\
 \theta' &=\epsilon \phi,\nonumber\\
 \phi' &=c_\pm (\theta)\hat F_N(\hat y,w),\nonumber\\
 v' &=\epsilon a(\theta,\phi)+q_\pm(\theta)\hat F_N(\hat y,w),\nonumber
\end{align}
with respect to the {\it fast time} $\tau = \epsilon^{-1} t$ where $()'=\frac{d}{d\tau}.$ This is a slow-fast system in non-standard form \cite{McClamroch1989}.  Only $\theta$ is truly slow whereas $(\hat y,w,\phi,v)$ are all fast. But the set of critical points
 \begin{align}
  C=\{(\hat y,w,\theta,\phi,v)\vert \hat y=0, w=0\},\eqlab{C}
 \end{align}
 for $\epsilon=0$ is just three dimensional. System \eqref{eq:haty} is PWS \cite{krihog,krihog2}. We now show that  \eqref{eq:haty}$_+$ contains stable and unstable manifolds $\gamma^{s,u}$ when the equivalent rigid body equations exhibit a Painlev\'e paradox, when $p_+(\theta_0)<0$. The saddle structure of $C$ within the fourth quadrant has been recognized before \cite{Lesuanan1990, DupontYamajako1997, SongKrausKumarDupont2001}.

 \begin{proposition}
 Consider the system \eqref{eq:haty}$_+$. Then for $p_+(\theta_0)<0$ there exist smooth stable and unstable sets $\gamma^{s,u}(\theta_0,\phi_0,v_0)$, respectively, of $(\hat y,w,\theta,\phi,v)=(0,0,\theta_0,\phi_0,v_0)\in C$ contained within $\hat F_N\ge 0$ given by
  \begin{align}
\gamma^{s,u}(\theta_0,\phi_0,v_0) = \bigg\{(\hat y,w,\theta,\phi,v)\vert \quad w &= \lambda_\mp\hat y,\quad \theta = \theta_0,\quad \phi = \phi_0-{c_+(\theta_0)\lambda_\mp^{-1} \left[1+\delta \lambda_\mp\right]}\hat y,\nonumber\\
v&=v_0+\frac{q_+(\theta_0)}{p_+(\theta_0)}\lambda_\mp(\theta_0)\hat y,\quad \hat y\le 0\bigg\},\eqlab{eq:unstableManifoldHaty}
  \end{align}
  with $\lambda_\mp(\theta_0)\lessgtr 0$ given in \eqref{eq:lambdapm} below.
\end{proposition}
 
\begin{proof}
Consider the smooth system, \eqref{eq:haty}$_{\hat F_N = -\hat y-\delta w}$, obtained from \eqref{eq:haty} by setting $\hat F_N=-\hat y-\delta w$. The linearization of \eqref{eq:haty}$_{\hat F_N = -\hat y-\delta w}$ about a point in $C$ then only has two 
non-zero eigenvalues:
\begin{align}
 \lambda_{\pm}(\theta)  &= -\frac{\delta p_+(\theta)}{2}\pm \frac12 \sqrt{\delta^2 p_+(\theta)^2 -4 p_+(\theta)},\eqlab{eq:lambdapm}
\end{align}
satisfying
\begin{align}
 \lambda_\pm^2 = -p_+(\theta) \left(1+\delta \lambda_\pm \right).\eqlab{eq:lambdapm2}
\end{align}
For $p_+(\theta)<0$ we have $\lambda_-<0<\lambda_+$. The eigenvectors associated with $\lambda_\pm$ are
$v_\pm = \left(1,\lambda_\pm,0,\frac{c_+}{p_+(\theta)}\lambda_\pm,\frac{q_+}{p_+(\theta)}\lambda_\pm\right)^T.$
Therefore the smooth system \eqref{eq:haty}$_{\hat F_N = -\hat y-\delta w}$ has a (stable, unstable) manifold $\gamma^{s,u}$ tangent to $v_\mp$ at $(\hat y,w,\theta,\phi,v)=(0,0,\theta_0,\phi_0,v_0)$. But then for $\hat y\le 0$, we have
 $\hat F_{N}(\hat y,\lambda_\pm \hat y) = -(1+\delta \lambda_\pm)\hat y = \frac{\lambda_\pm^2 }{p_+(\theta)}\hat y\ge 0$,
by \eqref{eq:lambdapm2}. Hence the restrictions of $\gamma^{s,u}$ in \eqref{eq:unstableManifoldHaty} to $\hat y<0$ are (stable, unstable) sets of $C$ for the PWS system \eqref{eq:haty}$_{\hat F_N= \left[-\hat y-\delta w\right]}$.
\end{proof}
\begin{remark}\remlab{CCalc}
%
  For the smooth system \eqref{eq:haty}$_{\hat F_N = -\hat y-\delta w}$, the critical manifold $C$ perturbs by Fenichel's theory \cite{fen1, fen2, fen3} to a smooth slow manifold $C_\epsilon$, being $C^\infty$ $\mathcal O(\epsilon)$-close to $C$. A simple calculation shows that 
$C_\epsilon:\, \hat y = \epsilon \frac{b(\theta,\phi)}{p_+(\theta)}(1+\mathcal O(\epsilon)),\, w=\mathcal O(\epsilon^2)$.
Since $b(\theta,\phi) <0$ in this case, $C_\epsilon\subset \{\hat y>0\}$ for $\epsilon$ sufficiently small. Therefore the manifold $C_\epsilon$ is invariant for the smooth system \eqref{eq:haty}$_{\hat F_N = -\hat y-\delta w}$ only. It is an artifact for the PWS system  \eqref{eq:haty}$_{\hat F_N= \left[-\hat y-\delta w\right]}$ since the square bracket vanishes for $\hat y>0$, by \eqref{eq:compliance}. 
\end{remark}


\begin{remark}\remlab{rem}
 Our arguments are geometrical and rely on hyperbolic methods of dynamical systems theory only. Therefore the results remain unchanged qualitatively if we replace the piecewise linear $\hat F_N$ in  \eqref{eq:hatFNHaty} with the nonlinear version 
  $\hat F_N(\hat y,w) = \left[h(\hat y,w)\right]$,
 where $h(\hat y,w)=-\hat y-\delta w+\mathcal O((\hat y+w)^2$ as in \eqref{hNon},
 having \eqref{hLin} as its linearization about $\hat y=w=0$. We would obtain again a saddle-type critical set $C$ with nonlinear (stable, unstable) manifolds $\gamma^{s,u}$. 
\end{remark}
Following the initial scaling \eqref{eq:initialscaling} of this section, we now consider the three phases of IWC.
\begin{figure}[h!] 
\begin{center}
\subfigure[]{\includegraphics[width=.495\textwidth]{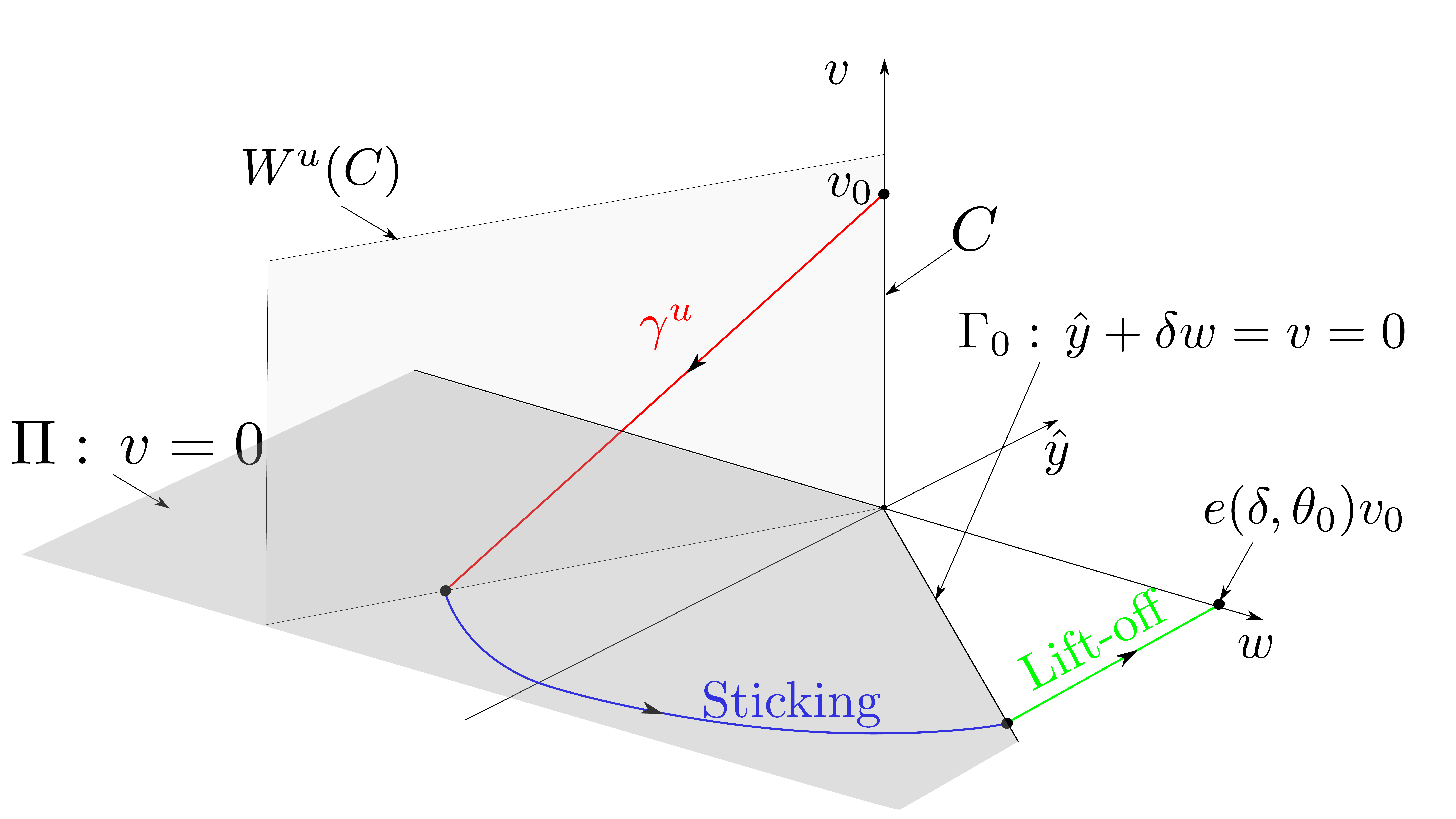}}
\subfigure[]{\includegraphics[width=.495\textwidth]{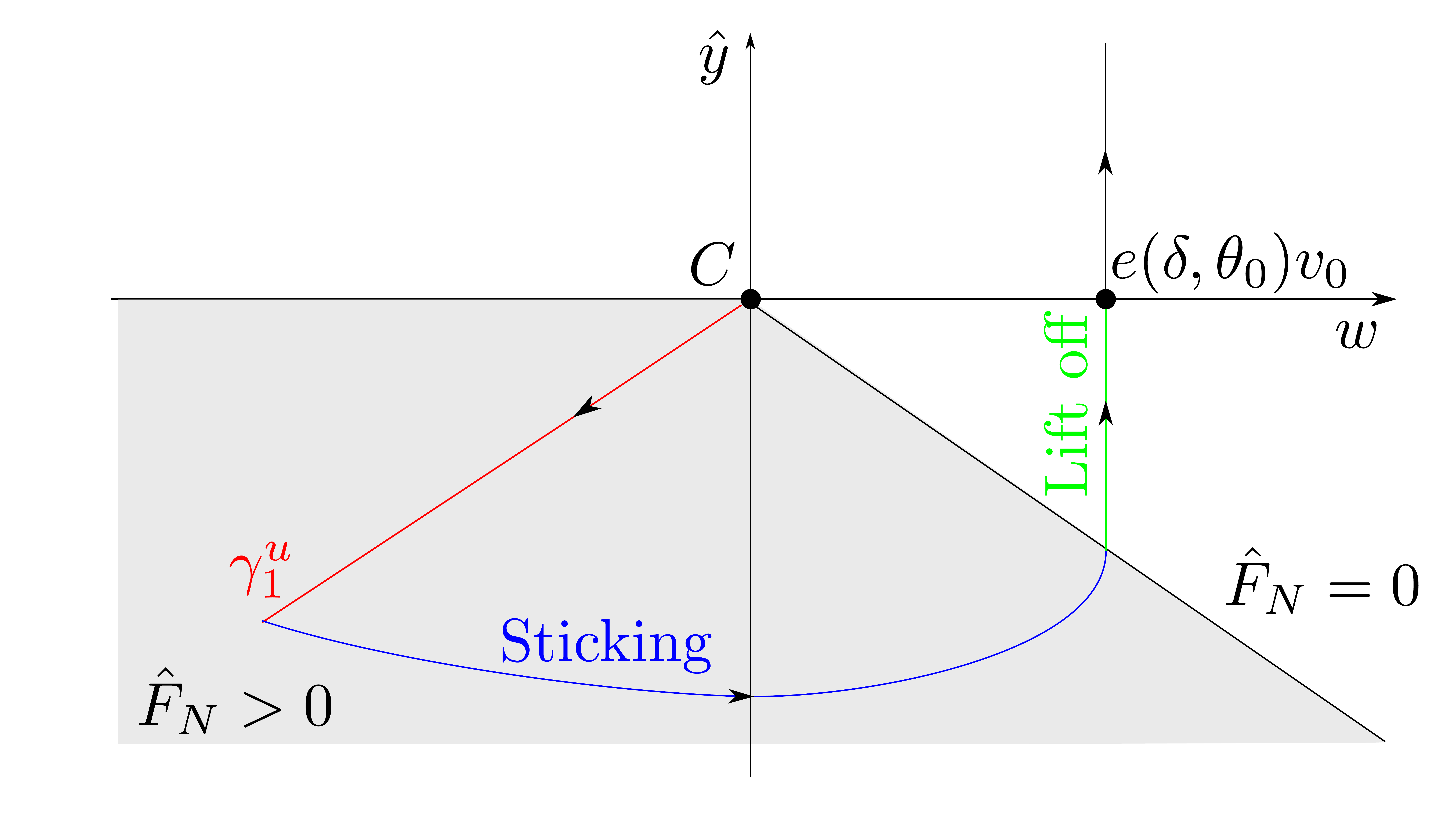}}
\end{center}
 \caption{The limit $\epsilon\rightarrow 0$ shown using (a)  the $(w,\hat y,v)$-variables and (b) a projection onto the $(w,\hat y)$-plane. The {\it slipping compression} phase, shown in red, where $\hat y$, $w$ and $v>0$ all decrease, is described geometrically by an unstable manifold $\gamma^u$ \eqref{eq:unstableManifoldHaty} of a critical set $C$, given in \eqref{C}. 
 The {\it sticking} phase (in blue) is described by Filippov \cite{filippov1988differential}. 
 Finally the {\it lift-off} phase (in green) occurs and we return to $\hat y=0$. In both figures the grey region is where $\hat F_N>0$.
 }
 \figlab{fig:IWC}
\end{figure}

\subsection{Slipping compression}\seclab{sec:slipping}

Now we describe the first phase of the regularized IWC: \textit{slipping compression}, which ends when $v=0$. We define the following section (or {\it switching manifold}), $\Pi$ shown in \figref{fig:IWC}(a):
\begin{align}
 \Pi=\{(\hat y,w,\theta,\phi,v)\vert v=0\}.\eqlab{eq:Pi2}
\end{align}
\propref{prop:slip} describes the intersection of the forward flow of initial conditions \eqref{eq:ics} with $\Pi$; in other words, the values of  $\hat y,w,\theta, \phi$ at the {\it end} of the slipping compression phase. 
\begin{proposition}\proplab{prop:slip}
 The forward flow of the initial conditions \eqref{eq:ics}
 under \eqref{eq:haty} intersects $\Pi$ in
 \begin{align}
 \gamma^u\cap \Pi+o(1)\equiv \bigg\{(\hat y,w,\theta,\phi,0)\vert \quad \hat y &= -\frac{p_+(\theta_0)}{q_+(\theta_0) \lambda_+(\theta_0)}v_0+ o(1),\quad w = -\frac{p_+(\theta_0)}{q_+(\theta_0)}v_0+ o(1) \nonumber\\
 \theta &=\theta_0+o(1),\quad \phi =\phi_0-\frac{c_+(\theta_0)}{q_+(\theta_0)}v_0 + o(1),\bigg\},\eqlab{eq:z2EpsPi2}
   \end{align}
   as $\epsilon\rightarrow 0$.   
\end{proposition}

\begin{remark}\remlab{here}
The $o(1)$-term in \eqref{eq:z2EpsPi2} is $\mathcal O(\epsilon^c)$ for any $c\in(0,1)$ (see also \lemmaref{finalcompress} below).
\end{remark}

%
%
\subsubsection{Proof of \propref{prop:slip}}\seclab{sec:propSlip}
We prove \propref{prop:slip} using Fenichel's normal form theory \cite{jones_1995}. 
But since \eqref{eq:haty}$_{\hat F_N= \left[-\hat y-\delta w\right]}$ is piecewise smooth, care must be taken. There are at least two ways to proceed. One way is to consider the smooth system \eqref{eq:haty}$_{\hat F_N=-\hat y-\delta w}$, then rectify $C_\epsilon$ by straightening out its stable and unstable manifolds. Then \eqref{eq:haty}$_{\hat F_N=-\hat y-\delta w}$ will be a standard slow-fast system to which Fenichel's normal form theory applies. Subsequently one would then have to ensure that conclusions based on the smooth \eqref{eq:haty}$_{\hat F_N=-\hat y-\delta w}$ also extend to the PWS system \eqref{eq:haty}$_{\hat F_N= \left[-\hat y-\delta w\right]}$. One way to do this is to consider the following scaling
\begin{align}
 \kappa_1:\quad \hat y = r_1 \hat y_1,\quad w= r_1 w_1,\quad \epsilon=r_1,\eqlab{eq:K2}
\end{align}
zooming in on $C$ at $\hat y=0,\,w=0$. In terms of the original variables, $y=\epsilon^2\hat y_1$, $w=\epsilon w_1$.  The scalings $(\hat y,w)$ and $(\hat y_1,w_1)$ have both appeared in the literature \cite{ChampneysVarkonyi2016, DupontYamajako1997, McClamroch1989}.

In this paper we follow another approach (basically reversing the process described above) which works more directly with the PWS system. Therefore in Section \secref{sec:kappa1} we study the scaling \eqref{eq:K2} first.
 We will show that the $(\hat y_1,w_1)$-system contains important geometry of the PWS system (significant, for example, for the separation of initial conditions in \thmref{cor}). Then in Section \secref{sec:kappa2} we connect the ``small'' ($\hat y=\mathcal O(\epsilon),w=\mathcal O(\epsilon)$) described by \eqref{eq:K2} with the ``large'' ($\hat y=\mathcal O(1),w=\mathcal O(1)$) in \eqref{eq:haty} by considering coordinates described by the following transformation:
  \begin{align}
 \kappa_2:&\quad \hat y = -r_2,\quad w = r_2 w_2,\quad \epsilon = r_2 \epsilon_2.\eqlab{eq:kappa1}
  \end{align}
  For $y_1<0$ we have the following coordinate change $\kappa_{21}$ between $\kappa_1$ and $\kappa_2$:
\begin{align}
 \kappa_{21}:\quad r_2 = -r_1 y_1,\quad w_2 = -w_1 y_1^{-1},\quad \epsilon_2 = -y_1^{-1}.\eqlab{eq:kappa21}
\end{align}
%
The coordinates in $\kappa_2$ \eqref{eq:kappa1} appear as a \textit{directional chart} obtained by setting $\bar y=-1$ in the blowup transformation $(r,\overline{\hat y},\bar w,\bar \epsilon)\mapsto (\hat y,w,\epsilon)$ given
by\footnote{More accurately, the chart $\bar y=-1$ corresponds to
\begin{align*}
 r = r_1\sqrt{1+\epsilon_1^2+w_1^2},\, \overline{\hat y} = -1/\sqrt{1+\epsilon_1^2+w_1^2},\, \bar w=w_1/\sqrt{1+\epsilon_1^2+w_1^2},\,\bar \epsilon=\epsilon_1/\sqrt{1+\epsilon_1^2+w_1^2},
\end{align*}
with $\overline{\hat y}^2+ \overline{w}^2+\overline{\epsilon}^2=1.$ See \cite{kuehn2015} for further details on directional and scaling charts.}
\begin{align}
 \hat y = r\bar {\hat y},\quad w=r\bar w,\quad \epsilon = r\bar \epsilon,\quad r\ge 0,\quad (\overline{\hat y},\bar w,\bar \epsilon)\in S^2=\{(\overline{\hat y},\bar w,\bar \epsilon)\vert \overline{\hat y}^2+\bar w^2+\bar \epsilon^2=1\}.\eqlab{eq:blowup}
\end{align}
The blowup is chosen so that the zoom in \eqref{eq:K2} coincides with the \textit{scaling chart} obtained by setting $\bar \epsilon=1$. The blowup transformation {\it blows up} $C$ to $\bar C:\, r=0,\,(\overline{\hat y},\bar w,\bar \epsilon)\in S^2$ a space $(\theta,\phi,v)\in \mathbb R^3$ of spheres.\footnote{Note that \eqref{eq:blowup} is not a {\it blowup} transformation in the sense of Krupa and Szmolyan \cite{krupa_extending_2001}, where geometric blowup is applied in conjunction with desingularization to study loss of hyperbolicity in slow-fast systems. We will not desingularize the vector-field here.}

The main advantage of our approach is that in chart $\kappa_2$ we can focus on $\overline C\cap \{\overline{\hat y}^{-1}\overline w>-\delta^{-1}, \,\overline{\hat y}<0\}$ (or simply $w_2<\delta^{-1}$ in \eqref{eq:kappa1}) of $\overline C$, the grey area in \figref{fig:IWC}, where 
\begin{align}
r^{-1} \hat F_N(\hat y,w) = \left[-\overline{\hat y} - \delta \overline{w}\right] = -\overline{\hat y}\left( 1+\delta \overline{\hat y}^{-1}\overline w \right)>0,\eqlab{CbarCalc}
\end{align}
and the system will be smooth. This enables us to apply Fenichel's normal form theory \cite{jones_1995} there. All the necessary patching for the PWS system is done independently in the scaling chart $\kappa_1$. Also chart $\kappa_2$ enables a matching between the two scalings that have appeared in the literature: {$\kappa_1\cap\{\hat y_1<0,\,r_1=0\}$, visible within $r_2=0$ of \eqref{eq:kappa1}, and system \eqref{eq:haty}$_{\epsilon=0}$, visible within $\epsilon_2=0$ of \eqref{eq:kappa1}}. 

\subsubsection{Chart $\kappa_1$}\seclab{sec:kappa1}
Let
  $\hat F_{N,1}(\hat y_1,w_1) = \epsilon^{-1} \hat F_N(\epsilon \hat y_1,\epsilon w_1) = \left[-\hat y_1-\delta w_1\right].$
Then applying chart $\kappa_1$ in \eqref{eq:K2} to the non-standard slow-fast system \eqref{eq:haty} gives the following equations:
\begin{align}
 {\hat y}_1' &= w_1,\eqlab{eq:haty2}\\
 w'_1 &=b(\theta,\phi)+p_+(\theta) \hat F_{N,1}(\hat y_1,w_1),\nonumber\\
 \theta' &=\epsilon \phi,\nonumber\\
 \phi' &=\epsilon c_+(\theta)\hat F_{N,1}(\hat y_1,w_1),\nonumber\\
 v' &=\epsilon \left(a(\theta,\phi)+q_+(\theta)\hat F_{N,1}(\hat y_1,w_1)\right),\nonumber
\end{align}

Equation \eqref{eq:haty2} is a slow-fast system in standard form: $(\hat y_1, w_1)$ are fast variables whereas $(\theta,\phi,v)$ are slow variables. By assumption \eqref{eq:nonExistence} of \thmref{thm:main}, $b<0$, $p_+<0$ and so, since $\hat F_{N,1}(\hat y_1,w_1)\ge 0$, we have $w_1'<0$ in \eqref{eq:haty2}. Hence there exists no critical set for the PWS system \eqref{eq:haty2}$_{\epsilon=0}$. The critical set $C_1$ of the {\it smooth} system \eqref{eq:haty2}$_{\hat F_{N,1}=-\hat y_1-\delta w_1}$, given by
\begin{align}
 C_1=\{(\hat y_1,w_1,\theta,\phi,v)\vert \hat y_1=\frac{b(\theta,\phi)}{p_+(\theta)},\,w_1 = 0\}, \eqlab{eqC1}
\end{align}
lies within $\hat y_1>0$. So  $C_1$ is an invariant of  \eqref{eq:haty2}$_{\hat F_{N,1}=-\hat y_1-\delta w_1}$ but an artifact of the PWS system \eqref{eq:haty2}$_{\hat F_{N,1}=[-\hat y_1-\delta w_1]}$, as shown in \figref{g}(a) (recall also \remref{CCalc}).
 
 \begin{figure}[h!] 
\begin{center}
\subfigure[]{\includegraphics[width=.495\textwidth]{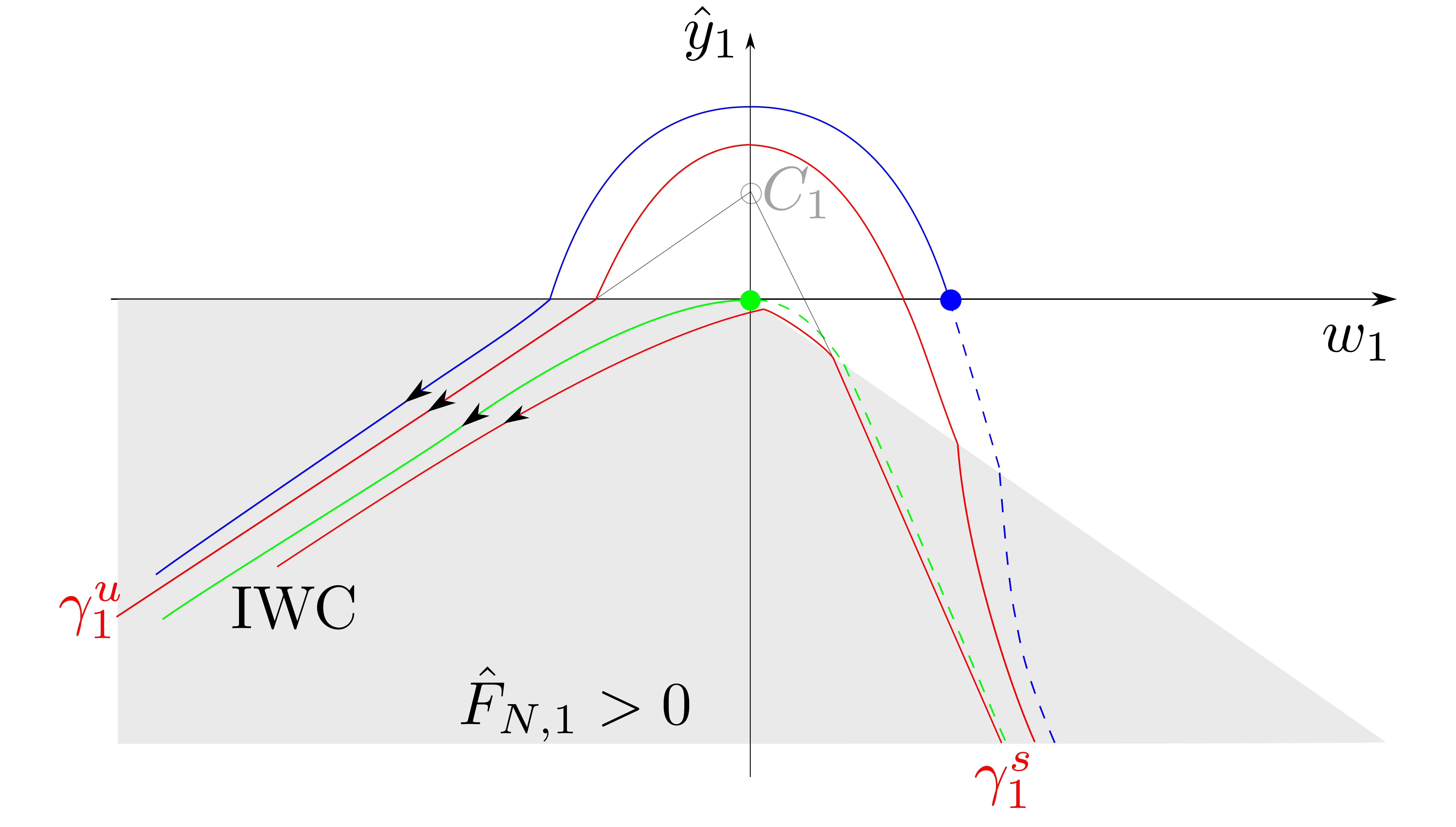}}
\subfigure[]{\includegraphics[width=.495\textwidth]{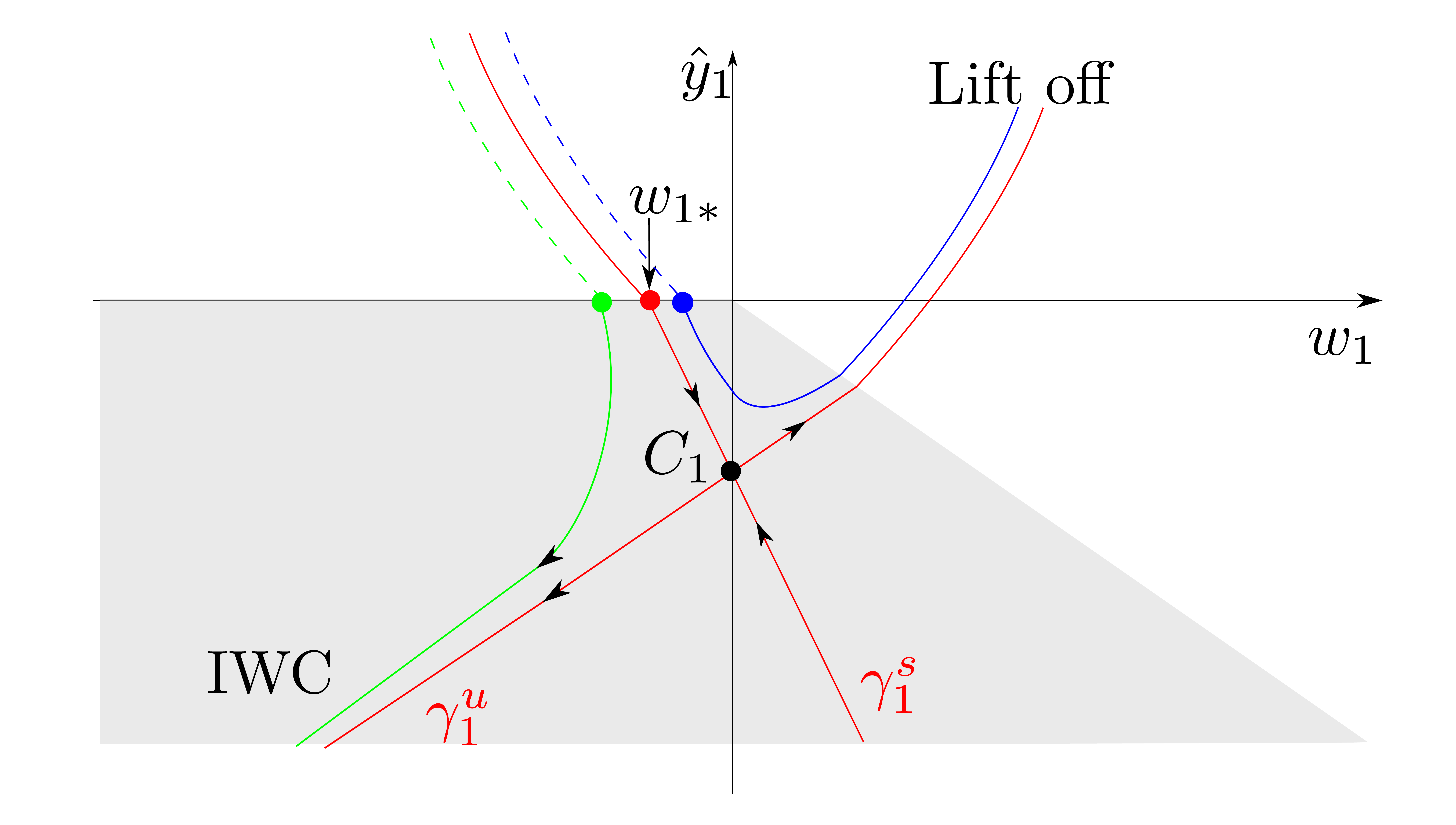}}
\end{center}
 \caption{(a) Phase portrait \eqref{eq:haty2}$_{\epsilon=0}$ for $b<0$ (\thmref{thm:main} in this Section). The critical set $C_1$  of \eqref{eq:haty2}$_{\hat F_{N,1}=-\hat y_1-\delta w_1}$, given by \eqref{eqC1}, is an artifact of the PWS system \eqref{eq:haty2}$_{\hat F_{N,1}=[-\hat y_1-\delta w_1]}$. (b) Phase portrait \eqref{eq:haty2}$_{\epsilon=0}$ for $b>0$ (\thmref{cor} in Section \secref{new}). Here $C_1$ is a saddle-type critical manifold for the PWS system, $\gamma_1^u$ is given by \eqref{eq:gammau1}, $\gamma_1^s$ by \eqref{eq:gammas1} and $w_{1*}$ by \eqref{eq:w1star}. The grey region is now where $\hat F_{N,1}>0$. Orbit segments outside this region  are parabolas turning downwards and upwards in (a) and (b), respectively. Dashed lines indicate backward orbits, from initial conditions on the $w_1$-axis. Similar figures from numerical computations appear in \cite{ChampneysVarkonyi2016}. }
 \figlab{g}
\end{figure}
 
The unstable manifold $ \gamma_1^u$ of $C_1$ in the smooth system \eqref{eq:haty2}$_{\hat F_{N,1}=-\hat y_1-\delta w_1}$ is given by
\begin{align}\eqlab{eq:gammau1}
 \gamma_1^u(\theta_0,\phi_0,v_0) = \bigg\{(\hat y_1,w_1,\theta_0,\phi_0,v_0)\vert \quad \hat y_1&=\frac{b(\theta_0,\phi_0)}{p_+(\theta_0)}+s,\quad w_1 =\lambda_+(\theta_0)s,\quad s\le -\frac{b(\theta_0,\phi_0)}{p_+(\theta_0)}\bigg\}
\end{align}
and its restriction to the subset $\hat y_1\le 0$, $w_1\le 0$ where $\hat F_{N,1}\ge 0$, is locally invariant for the PWS system \eqref{eq:haty2}$_{\hat F_{N,1}=\left[-\hat y_1-\delta w_1\right]}$.

In chart $\kappa_1$, initial conditions \eqref{eq:ics} now become:  
\begin{align}
  (\hat y_1,w_1,\theta, \phi,v) = (0,\mathcal O(1),\theta_0,\phi_0,v_0). \eqlab{eq:hat1ICs}
\end{align}
In \lemmaref{lemma1}, we determine the values of the variables during the slipping compression phase, starting from initial conditions \eqref{eq:hat1ICs}, as seen in chart $\kappa_1$, and show that the system remains close to  $\gamma_1^u(\theta_0,\phi_0,v_0)$.

\begin{lemma}\lemmalab{lemma1}
Consider 
$ \Lambda_1=\{(\hat y_1,w_1,\theta,\phi,v)\vert \hat y_1 = -\nu^{-1}\}$
with $\nu>0$ small. Then the forward flow of \eqref{eq:hat1ICs} under \eqref{eq:haty2} intersects $\Lambda_1$ in 
\begin{align}\eqlab{eq:z1}
 z_{1}(\epsilon) \equiv (-\nu^{-1},w_{1c}(\nu)+\mathcal O(\epsilon),\theta_0+\mathcal O(\epsilon),\phi_0+\mathcal O(\epsilon),v_0+\mathcal O(\epsilon)).
\end{align}
where 
$ w_{1c} (\nu)= -\lambda_+ \nu^{-1}(1+o(1)),\quad \nu\rightarrow 0.$
\end{lemma}
\begin{proof}
 Consider the layer problem \eqref{eq:haty2}$_{\epsilon=0}$. Then $\theta(\tau)=\theta_0,\,\phi(\tau)=\phi_0,\,v(\tau)=v_0$. {Since $b<0$, initial conditions \eqref{eq:hat1ICs} with $w_1>0$ return to $\hat y_1=0$ with $w_1<0$, see \figref{g}(a).} Therefore we consider $w_1(0)\le 0$ subsequently. Now we solve \eqref{eq:haty2}$_{\hat F_{N,1}=-\hat y_1-\delta w_1,\,\epsilon=0}$ to find
 \begin{align}
 \hat y_1(\tau) &= \frac{b(\theta_0,\phi_0)}{p_+(\theta_0)}+k_+ e^{\lambda_+ \tau} +k_- e^{\lambda_-\tau},\eqlab{eq:haty2Sol}\\
 w_1(\tau)&=k_+ \lambda_+ e^{\lambda_+ \tau} +k_- \lambda_-e^{\lambda_-\tau}.\nonumber
\end{align}
Here $k_+ = \frac{\lambda_- b/p_+-w_1(0)}{\lambda_+-\lambda_-}<0, k_- = \frac{w_1(0)-\lambda_+ b/p_+}{\lambda_+-\lambda_-}<0$ depend on initial conditions. 
Since $\lambda_-<0<\lambda_+$ the solution remains within $\hat F_{N,1}>0$ for $\tau>0$ and contracts towards $\gamma_1^u(\theta_0,\phi_0,v_0)$ as $\tau\rightarrow \infty$. Therefore there exists a time $\tau_{c}=\tau_{c}(\nu)>0$ such that  $\Lambda_1$ is reached. Then at $\tau=\tau_{c}$, we have
 $z_{1}(0) = (-\nu^{-1},w_{1c},\theta_0,\phi_0,v_0),$
where
  $w_{1c}(\nu)\equiv w_1(\tau_{c}) = k_+ \lambda_+ e^{\lambda_+ \tau_{c}} +k_- \lambda_-e^{\lambda_-\tau_{c}}=-\lambda_+ \nu^{-1}(1+o(1))<0$.
 To obtain $z_1(\epsilon)$ we apply regular perturbation theory and the implicit function theorem using transversality to $\Lambda_1$ for $\epsilon=0$.
\end{proof}

For $\epsilon>0$ the variables $(\phi,v)$ will vary by $\mathcal O(1)$-amount as $\hat y_1,w_1\rightarrow-\infty$. But the variables $(\phi,v)$ are fast in \eqref{eq:haty} and slow in \eqref{eq:haty2}. To describe this transition we change to chart $\kappa_2$.

\subsubsection{Chart $\kappa_2$}\seclab{sec:kappa2}
Writing the non-standard slow-fast PWS system \eqref{eq:haty}$_{\hat F_N=\left[-\hat y-\delta w\right]}$ in chart $\kappa_2$, given by \eqref{eq:kappa1}, gives the following smooth (as anticipated by \eqref{CbarCalc}) system
\begin{align}
\epsilon_2' &=\epsilon_2 w_2,\eqlab{eq:eqnK1}\\
  w_2' &=\epsilon_2 b(\theta,\phi) +p_+(\theta) \left(1-\delta w_2\right)+w_2^2,\nonumber\\
  \theta'&=\epsilon \phi,\nonumber\\
  \phi'&=c_+(\theta)r_2\left(1-\delta w_2\right),\nonumber\\
  v' &=\epsilon a(\theta,\phi)+q_+(\theta)r_2 \left(1-\delta w_2\right),\nonumber\\
  r_2'&=-r_2w_2,\nonumber
 \end{align}
  on the box
   $U_2 = \{(\epsilon_2,w_2,\theta,\phi,v,r_2)\vert \epsilon_2\in [0,\nu],\,w_2\in [-\lambda_+-\rho,-\lambda_++\rho],\,r_2\in [0,\nu]\},$
 for $\rho>0$ sufficiently small (so that $w_2<\delta^{-1}$) and $\nu$ as above. Notice that $z_1(\epsilon)$ from \eqref{eq:z1} in chart $\kappa_2$ becomes
\begin{align}
 z_2(\epsilon) \equiv \kappa_{21}(z_1(\epsilon)):\quad r_2 = \epsilon \nu^{-1},\quad w_2 = w_{1c}\nu +\mathcal O(\epsilon),\quad \epsilon_2 = \nu,\eqlab{z2Eps}
\end{align}
using \eqref{eq:kappa21}. 
 Clearly $z_2(\epsilon)\in \kappa_{21}(\Lambda_1)\subset U_2$, $\kappa_{21}(\Lambda_1)$ being the face of the box $U_2$ with $\epsilon_2=\nu$. In this section we will for simplicity write subsets such as $\{(\epsilon_2,w_2,\theta,\phi,v,r_2)\in U_2\vert\cdots\}$ by $\{U_2\vert \cdots\}$.
  \begin{lemma}\lemmalab{lemma2}
 The set 
 $M_2=\{U_2\vert \quad r_2=0,\,\epsilon_2=0, \, w_2=-\lambda_+\}$
 is a set of critical points of \eqref{eq:eqnK1}. 
 Linearization around $M_2$ gives only three non-zero eigenvalues
 $ -\lambda_+<0,\,\lambda_--\lambda_+<0,\,\lambda_+>0,$
 and so $M_2$ is of saddle-type. The stable manifold is $W^s(M_2)=\{U_2\vert r_2=0\}$ while the unstable manifold is $W^u(M_2)=\{U_2\vert \epsilon_2=0,\,w_2=-\lambda_+\}$.
In particular, the $1D$ unstable manifold $\gamma^u_{2}(\theta_0,\phi_0,v_0)\subset W^u(M_2)$ of the base point $(\epsilon_2,w_2,\theta,\phi,v,r_2)= (0,0,\theta_0,\phi_0,v_0,0)\in M_2$ is given by 
 \begin{align}
 \gamma^u_{2}(\theta_0,\phi_0,v_0) =\bigg\{U_2\vert \quad w_2&= -\lambda_+(\theta_0),\quad \theta =\theta_0,\quad \phi = \phi_0- \frac{c_+(\theta_0)}{p_+(\theta_0)} \lambda_+(\theta_0)r_2,\eqlab{eq:gammau2}\\ v &= v_0- \frac{q_+(\theta_0)}{p_+(\theta_0)} \lambda_+(\theta_0)r_2,\quad r_2\ge 0,\quad \epsilon_2 =0\nonumber\bigg\}.
 \end{align}
\end{lemma}
\begin{proof}
 The first two statements follow from straightforward calculation. For $\gamma^u_{2}(\theta_0,\phi_0,v_0)$, we restrict to the invariant set: $\epsilon_2=0$, $w_2=-\lambda_+$ and solve the resulting reduced system. 
\end{proof}
\begin{remark}
 Notice that the set $\gamma^u_{2}(\theta_0,\phi_0,v_0)$ is just $\gamma^u(\theta_0,\phi_0,v_0)$ in \eqref{eq:unstableManifoldHaty} written in chart $\kappa_2$ for $\epsilon_2=0$. 
\end{remark}

Notice that $z_2(0)\subset W^s(M_2)$. 
The forward flow of $z_2(0)$ is described for $\tau\ge \tau_{c}$ by writing solution \eqref{eq:haty2Sol} to the layer problem \eqref{eq:haty2}$_{\epsilon=0}$ in chart $\kappa_2$ using $\kappa_{21}$, to get
\begin{align}
 \epsilon_2(\tau) &= -\left(\frac{b(\theta_0,\phi_0)}{p_+(\theta_0)}+k_+ e^{\lambda_+ \tau} +k_- e^{\lambda_-\tau}\right)^{-1}\nonumber\\
 &=-k_+^{-1}e^{-\lambda_+ \tau}(1+\mathcal O(e^{-\lambda_+\tau} + e^{(\lambda_--\lambda_+)\tau})),\quad \tau\rightarrow \infty\eqlab{eq:eos2w2Sol}\\
 w_2(\tau)&=-\left(k_+ \lambda_+ e^{\lambda_+ \tau} +k_- \lambda_-e^{\lambda_-\tau}\right)\left(\frac{b(\theta_0,\phi_0)}{p_+(\theta_0)}+k_+ e^{\lambda_+ \tau} +k_- e^{\lambda_-\tau}\right)^{-1} \nonumber\\
 &= -\lambda_+ (1+\mathcal O(e^{-\lambda_+\tau} + e^{(\lambda_--\lambda_+)\tau})), \quad \tau\rightarrow \infty.\nonumber
\end{align}
In the subsequent lemma we follow $z_2(\epsilon)\subset \{\epsilon_2=\nu\}$ up until $r_2=\nu$, with $\nu$ sufficiently small, by applying Fenichel's normal form theory. 
\begin{lemma}\lemmalab{finalcompress}
Let $c\in (0,1)$ and set $\Lambda_2 = \{U_2\vert r_2=\nu\}$. Then for $\nu$ and $\rho$ sufficiently small, the forward flow of $z_2(\epsilon)$ in \eqref{z2Eps} intersects $\Lambda_2$ in
 \begin{align}
  \bigg\{\Lambda_2\vert \quad w_2&=-\lambda_++\mathcal O(\epsilon^c),\quad \theta=\theta_0+\mathcal O(\epsilon \ln \epsilon^{-1}), \quad \phi =\phi_0- \frac{c_+(\theta_0)}{p_+(\theta_0)}\lambda_+(\theta_0)\nu+ \mathcal O(\epsilon^c),\nonumber\\
  v&=v_0- \frac{q_+(\theta_0)}{p_+(\theta_0)}\lambda_+(\theta_0)\nu+ \mathcal O(\epsilon^c)\bigg\}. \eqlab{eq:eos2w2SolNew}
 \end{align}
%
 as $\epsilon\rightarrow 0$.
  
\end{lemma}
\begin{proof}
By Fenichel's normal form theory we can make the slow variables independent of the fast variables $(\epsilon_2,w_2,r_2)$:
\begin{lemma}\lemmalab{straight}
For $\nu$ and $\rho$ sufficiently small, then within $U_2$ there exists a smooth transformation $(\epsilon_2,w_2,\phi,v,r_2)\mapsto (\tilde \phi,\tilde v)$ satisfying
\begin{align}
\tilde \phi &=\phi+\frac{c_+(\theta)}{p_+(\theta)}\lambda_+(\theta)r_2+ \mathcal O(r_2 (w_2+\lambda_+)),\eqlab{eq:Tilde}\\
\tilde v&=v+\frac{q_+(\theta)}{p_+(\theta)}\lambda_+(\theta)r_2+ \mathcal O(r_2 (w_2+\lambda_+)+\epsilon),\nonumber
 \end{align}
 which transforms \eqref{eq:eqnK1} into
\begin{align}
  {\epsilon}_2' &={\epsilon}_2 w_2,\eqlab{eq:eqnK1Tilde}\\
  w_2' &=\epsilon_2 b(\theta,\tilde \phi) +p_+(\theta) \left(1-\delta w_2\right)+w_2^2+\mathcal O(\epsilon),\nonumber\\
  \theta' &=\epsilon \tilde \phi,\nonumber\\
 {\tilde \phi}'&=0,\nonumber\\
 {\tilde v}' &=0,\nonumber\\
 {r}_2'&=-r_2w_2.\nonumber
 \end{align}
 \end{lemma}
\begin{proof}
Replace $r_2$ by $\nu r_2$ in \eqref{eq:eqnK1} and consider $\nu$ small. Then $\epsilon_2=r_2=0$, $w_2=-\lambda_+$ is a saddle-type slow manifold for $\nu$ small. The result then follows from Fenichel's normal form theory \cite{jones_1995}. Using $\phi=\tilde \phi+\mathcal O(r_2)$ together with $r_2\epsilon_2=\epsilon$ in the $w_2$-equation then gives the desired result.
\end{proof}
To prove \lemmaref{finalcompress} we then integrate the normal form \eqref{eq:eqnK1Tilde} with initial conditions $z_2(\epsilon)$ from \eqref{z2Eps} from (a reset) time $\tau=0$ up to $\tau=T$, defined implicitly by $r_2(T)=\nu$. 
Clearly 
 $\theta(T) =\theta_0+\mathcal O(\epsilon T)$,
$ \tilde \phi(T)=\tilde \phi_0$,
$ \tilde v(T)=\tilde v_0$.
Then, from \eqref{eq:eos2w2Sol}, Gronwall's inequality and the fact that $1-\lambda_-\lambda_+^{-1}>1$, we find
\begin{align}
T&=\lambda_+^{-1}\ln \epsilon^{-1} (1+o(1)) \eqlab{eq:TT}\\
 \epsilon_2(T) &= \epsilon \nu^{-1} \nonumber\\
 w_2(T)&=-\lambda_+ (1+\mathcal O(e^{-\lambda_+T} + e^{(\lambda_--\lambda_+)T}+\epsilon)) = -\lambda_+ +\mathcal O(\epsilon^{c(1-\lambda_-\lambda_+^{-1})} +\epsilon^{c})=-\lambda_++\mathcal O(\epsilon^{c}),
\end{align}
for $c\in (0,1)$. Then we obtain the expressions for $\theta=\theta(T)$, $\phi=\phi(T)$ and $v=v(T)$ in \eqref{eq:eos2w2SolNew} from \eqref{eq:Tilde} 
in terms of the original variables. 
\end{proof}

\subsubsection{Completing the proof of \propref{prop:slip}}

To complete the proof of \propref{prop:slip} we then return to \eqref{eq:haty} using \eqref{eq:kappa1} and integrate initial conditions \eqref{eq:eos2w2SolNew} within $\{\hat y=-r_2=-\nu\}$, up to $\Pi:\,v=0$ given in \eqref{eq:Pi2}, using regular perturbation theory and the implicit function theorem. This gives \eqref{eq:z2EpsPi2}
which completes the proof of \propref{prop:slip}.

\subsection{Sticking}\seclab{sec:stick}
After the slipping compression phase of the previous section, the rod then sticks on the sliding manifold $\Pi$ given in \eqref{eq:Pi2}, with $(\hat y,w,\theta,\phi)$ given by \eqref{eq:z2EpsPi2}. This is a corollary of the following lemma:
\begin{lemma}\lemmalab{lemma:nonStick}
 Suppose $a\ne 0$, $q_+<0$, $q_->0$. Consider the (negative) function
 \begin{align*}
  \mathcal F(\theta,\phi) =  \left\{\begin{array}{cc}
                                                                   \frac{a(\theta,\phi)}{q_+(\theta)}\quad \text{if}\quad a>0,\\
                                                                   \frac{a(\theta,\phi)}{q_-(\theta)}\quad \text{if}\quad a<0.
                                                                  \end{array}\right.
                                                                  \end{align*}
 Then there exists a set of visible folds at:
 \begin{align}
 \Gamma_\epsilon \equiv \{(\hat y,w,\theta,\phi,v)\in \Pi\vert \quad \hat y + \delta w =\epsilon \mathcal F(\theta,\phi)\},\eqlab{vFold}
 \end{align}
 of the Filippov system \eqref{eq:haty},
dividing the switching manifold $\Pi:\,v=0$ into (stable) sticking:
 $\Pi_s\equiv \{(\hat y,w,\theta,\phi,v)\in \Pi\vert \quad \hat y + \delta w <\epsilon \mathcal F(\theta,\phi)\},$
and crossing upwards (downwards) for $a>0$ ($a<0$):
$\Pi_c\equiv \{(\hat y,w,\theta,\phi,v)\in \Pi\vert \quad \hat y + \delta w>\epsilon \mathcal F(\theta,\phi)\}.$ 
\end{lemma}
\begin{proof}
Simple computations, following \cite{filippov1988differential}; see also \propref{lem:Filippov}.
\end{proof}

%
The forward motion of \eqref{eq:z2EpsPi2} within $\Pi_s\subset \Pi$ for $\epsilon\ll 1$ is therefore subsequently described by the Filippov vector-field \eqref{eq:Filippov} in \propref{lem:Filippov}
   \begin{align}
 {\hat y}'&=w,\eqlab{eq:hatStick}\\
  w' &=\epsilon b(\theta,\phi)+S_w(\theta) \left[-\hat y-\delta w\right],\nonumber\\
  \theta' &=\epsilon \phi,\nonumber\\
  \phi' &=S_\phi(\theta)\left[-\hat y-\delta w\right],\nonumber
 \end{align}
here written in terms of $\hat y$ and the fast time $\tau$, until sticking ends at the visible fold $\Gamma_\epsilon$.  Note this always occurs for $0<\epsilon\ll 1$ since 
${\hat y}''=w' >0$,
for $\left[-\hat y-\delta w\right]>0$.
 
We first focus on $\epsilon=0$. From \eqref{eq:hatStick}, $\theta=\theta_0$, a constant, and
 \begin{align}
 {\hat y}'&=w,\eqlab{eq:stickEps0}\\
  w' &=S_w(\theta) \left[-\hat y-\delta w\right],\nonumber\\
  \phi' &=S_\phi(\theta)\left[-\hat y-\delta w\right].\nonumber
 \end{align}
 We now integrate \eqref{eq:stickEps0}, using \eqref{eq:z2EpsPi2} for $\epsilon=0$ as initial conditions, given by 
 \begin{align}
(\hat y(0),w(0),\phi(0)) = \left(-\frac{p_+(\theta_0)}{q_+(\theta_0) \lambda_+(\theta_0)}v_0,-\frac{p_+(\theta_0)}{q_+(\theta_0)} v_0,\phi_0-\frac{c_+(\theta_0)}{q_+(\theta_0)}v_0\right),\eqlab{eq:initStick}
   \end{align}
up until the section
 $\Gamma_0:\hat y +\delta w=0$
shown in \figref{fig:IWC}(a), where sticking ceases for $\epsilon=0$, by \lemmaref{lemma:nonStick} and \eqref{vFold}$_{\epsilon=0}$. 
%
%
We then obtain a function $e(\delta,\theta_0)>0$ in the following proposition, which relates the horizontal velocity at the start of the slipping compression phase $v_0$ \eqref{eq:ics} with the values of $(\hat y,w,\phi)$ on $\Gamma_0$, at the end of the sticking phase.
 \begin{proposition}\proplab{prop:stick}
 There exists a smooth function $e(\delta,\theta_0)>0$ and a time $\tau_s>0$ such that: .
 $(\hat y(\tau_s),w(\tau_s),\phi(\tau_s))\in \Gamma_0$ with
 \begin{align*}
   \hat y(\tau_s) &=-\delta e(\delta,\theta_0)v_0,\\
   w(\tau_s) &= e(\delta,\theta_0)v_0,\\
    \phi(\tau_s) &=\phi_0+\left\{-\frac{c_+(\theta_0)}{q_+(\theta_0)}+\frac{S_\phi(\theta_0)}{S_w(\theta_0)}\left(e(\delta,\theta_0)+\frac{p_+(\theta_0)}{q_+(\theta_0)}\right)\right\}v_0,
  \end{align*}
 where $(\hat y(\tau),w(\tau),\phi(\tau))$ is the solution of \eqref{eq:stickEps0} with initial conditions \eqref{eq:initStick}. 
 The function $e(\delta,\theta_0)$ 
is monotonic in $\delta$:
 $\partial_{\delta} e(\delta,\theta_0)<0$,
and satisfies \eqsref{eDeltaLarge}{eDeltaSmall} for $\delta\gg 1$ and $\delta \ll 1$, respectively.

 \end{proposition}

\begin{proof}
The existence of $\tau_s$ is obvious. Linearity in $v_0$ follows from \eqref{eq:initStick} and the linearity of \eqref{eq:stickEps0}. Since $\dot{\hat y}=w$, we have $e>0$. The $\phi$-equation follows since 
 $\phi' = \frac{S_\phi(\theta)}{S_w(\theta)}w'$.
The monotonicity of $e$ as function $\delta$ is the consequence of simple arguments in the $(w,\hat y)$-plane using \eqref{eq:stickEps0} and the fact that $w(0)$ in \eqref{eq:initStick} is independent of $\delta$ while $\hat y(0)=\hat y_0(\delta)$ decreases (since $\lambda_+$ is an increasing function of $\delta$). 
To obtain the asymptotics we first solve \eqref{eq:stickEps0} with $\delta\ne \frac{2}{\sqrt{S_w(\theta_0)}}$. 
Simple calculations show that
\begin{align}
e(\delta,\theta_0) = \frac{\xi_+}{ \xi_-}\left(\lambda_+-\xi_-\right)\frac{p_+}{q_+\lambda_+} e^{\xi_+\tau_s},\eqlab{eq:eGeneral}
\end{align}
suppressing the dependency on $\theta_0$ on the right hand side, 
where 
  $\xi_\pm = -\frac{\delta S_w}{2}\pm \frac12 \sqrt{\delta^2 S_w^2-4S_w}$,
and $\tau_s$ is the least positive solution of
 $e^{(\xi_+-\xi_-) \tau_s} = \frac{\xi_-^2(\lambda_+ -\xi_+)}{\xi_+^2(\lambda_+-\xi_-)}$.
For $\delta\gg 1$ the eigenvalues $\xi_\pm$ are real and negative. Hence
 $\tau_s = \frac{1}{\xi_+-\xi_-}\ln \left(\frac{\xi_-^2(\lambda_+ -\xi_+)}{\xi_+^2(\lambda_+-\xi_-)}\right)$.
Now using 
 $\xi_+ = -S_w \delta(1+\mathcal O(\delta^{-2}),\quad \xi_- = \frac{S_w}{\xi_-} = -\delta^{-1}(1+\mathcal O(\delta^{-1})$,
and
$ \lambda_+ = -p_+\delta (1+\mathcal O(\delta^{-2})),$
we obtain
$ \xi_+ \tau_s = \mathcal O(\delta^{-2} \ln \delta^{-1}),$
and hence
\begin{align}
 e(\delta,\theta_0) = -\frac{S_w-p_+}{ q_+ S_w}\delta^{-2}\left(1+\mathcal O(\delta^{-2} \ln \delta^{-1})\right),\eqlab{eq:eLarge}
\end{align}
as $\delta\rightarrow \infty$. For $\delta\ll 1$, $\xi_\pm$ are complex conjugated with negative real part. This gives
 $\tau_s = \frac{2i}{\xi_+-\xi_-)} \left(\phi-\pi n\right),\, \phi = \text{arg}\,((\lambda_+-\xi_+)\xi_-^2)>0,\, n = \lfloor \phi/\pi\rfloor.$
Using the asymptotics of $\xi_\pm $ and $\lambda_+$ we obtain
 $\tau_s = \frac{\pi-\text{arctan}\left(\sqrt{-\frac{S_w}{p_+}}\right)}{\sqrt{S_w}}-\frac12 \delta(1+\mathcal O(\delta)),$
and then 
\begin{align}
 e(\delta,\theta_0) = -\frac{\sqrt{p_+(p_+-S_w)}}{q_+} \left(1-\frac{\sqrt{S_w}}{2}\left(\pi-\text{arctan}\left(\sqrt{-\frac{S_w}{p_+}}\right) \right)\delta+\mathcal O(\delta^2)\right),\eqlab{eq:eLow}
\end{align}
as $\delta\rightarrow 0^+$.
Simple algebraic manipulations of \eqsref{eq:eLarge}{eq:eLow} using \eqref{eq:Sw} give the expressions in \eqsref{eDeltaLarge}{eDeltaSmall}.

\end{proof}

\begin{remark}
 The critical value $\delta = \delta_{crit}(\theta_0) \equiv \frac{2}{\sqrt{S_w(\theta_0)}}$ gives a double root of the characteristic equation. Note that $\delta_{crit}(\frac{\pi}{2})=2$ for the classical Painlev\'e problem, as expected (see section \secref{sec:compliancebasedmethod}).
\end{remark}
For $0<\epsilon \ll 1$ sticking ends along the visible fold at $\Gamma_\epsilon$. We therefore perturb from $\epsilon=0$ as follows:
\begin{proposition}\proplab{prop:stickNew} 
 The forward flow of \eqref{eq:z2EpsPi2} under the Filippov vector-field \eqref{eq:hatStick} intersects the set of visible folds $\Gamma_\epsilon$ $o(1)$-close to the intersection of \eqref{eq:z2EpsPi2}$_{\epsilon=0}$ with $\Gamma_0$ described in \propref{prop:stick}.
\end{proposition}
\begin{proof}
 Since the $\epsilon=0$ system is transverse to $\Gamma_0$ we can apply regular perturbation theory and the implicit function theorem to perturb $\tau_s$ continuously to $\tau_s+o(1)$. The result then follows.
\end{proof}

%


\subsection{Lift-off}\seclab{sec:lift}
Beyond $\Gamma_\epsilon$ we have $\hat F_N\equiv 0$ and lift-off occurs. For $\epsilon=0$ we have
${\hat y}' =w$
 and $w'= \theta '= \phi' =v'=0$. By \propref{prop:stickNew} and regular perturbation theory, we obtain the desired result in \thmref{thm:main}. In terms of the original (slow) time $t$, it follows that the time of IWC is of order $\mathcal O(\epsilon \ln \epsilon^{-1})$ (recall \eqref{eq:TT}). As $\epsilon\rightarrow 0$, IWC occurs instantaneously, as desired. 


\section{Proof of \thmref{cor}: IWC in the indeterminate case}\seclab{new}
Here, by assumption \eqref{eq:nonUnique}, we have $b>0$. Now $C_1\subset \{\hat y_1<0\}$, where $C_1$ is given by \eqref{eqC1}; see also \figref{g}(b). The stable manifold of $C_{1}\cap\{\theta=\theta_0,\,\phi=\phi_0,\,v=v_0\}$ is:
 \begin{align}\eqlab{eq:gammas1}
\gamma_1^s(\theta_0,\phi_0,v_0)= \bigg\{(\hat y_1,w_1,\theta_0,\phi_0,v_0)\vert \quad \hat y_1&=\frac{b(\theta_0,\phi_0)}{p_+(\theta_0)}+s,\quad w_1 =\lambda_-(\theta_0)s,\quad s\le -\frac{b(\theta_0,\phi_0)}{p_+(\theta_0)},\bigg\}
 \end{align}
 with $\lambda_-$ defined in \eqref{eq:lambdapm}. $\gamma_1^s$ intersects the $w_1$-axis in
 \begin{align}\eqlab{eq:w1star}
  \gamma_1^s\cap \{\hat y_1=0\}:\quad w_1 =w_{1*}\equiv -\lambda_-(\theta_0)\frac{b(\theta_0,\phi_0)}{p_+(\theta_0)}<0,
 \end{align}
 and divides the negative $w_1$-axis into (i) initial conditions that lift off directly ($w_{10}>w_{1*}$, blue in \figref{g}(b)) and (ii) initial conditions that undergo IWC before returning to $\hat y=0$ ($w_{10}<w_{1*}$, green in \figref{g}(b)). (A canard phenomenon occurs around $w_{10}=w_{1*}$ for $0<\epsilon\ll 1$ where the solution follows a saddle-type slow manifold for an extended period of time.) For $w_{10}<w_{1*}$ the remainder of the proof of \thmref{cor} on IWC in the indeterminante case then follows the proof of \thmref{thm:main} above.


\section{Discussion}\seclab{discussion}

The quantity $e(\delta,\theta_0)$ relates the initial horizontal velocity $v_0$ of the rod to the resulting vertical velocity at the end of IWC. It is like a ``horizontal coefficient of restitution''. The leading order expression of $e(\delta,\theta_0)$ in \eqref{eDeltaLarge}  for $\delta\gg 1$ is independent of $\mu$, in general. Using the expressions for $q_\pm$ and $p_\pm$ in \eqref{eq:coeffs}, together with \eqref{eq:eLarge}, we find for large $\delta$ that
 \begin{align}
  e(\delta,\theta_0) &=\frac{\alpha }{2(1+\alpha)}\sin (2\theta_0)\delta^{-2}\left(1+\mathcal O(\delta^{-2} \ln \delta^{-1})\right),\quad \theta_0\in (\theta_1,\theta_2).\eqlab{eq:eDeltaLarge}
 \end{align}
The limit $\delta\rightarrow \infty$ is not uniform in $\theta\in (\theta_1,\theta_2)$. 

The expression for $\delta\ll 1$ is more complicated and {\it does} depend upon $\mu$, in general. Using \eqref{eq:coeffs} and \eqref{eq:eLow}, for $\delta=0$, we have:
\begin{align}\eqlab{eq:eDeltaZero}
 e(0,\theta_0)= \sqrt{\frac{(1+\alpha\cos^2\theta_0-\mu\alpha\sin\theta_0\cos\theta_0)}{(\alpha\sin\theta_0\cos\theta_0-\mu(1+\alpha\sin^2\theta_0))}\frac{\alpha\sin\theta\cos\theta_0}{(1+\alpha\sin^2\theta_0)}}
\end{align}
We plot $e(0,\theta_0)$ in \figref{fig:eDelta0}(a) for $\alpha=3$ and $\mu=1.4$. \figref{fig:eDelta0}(b) shows the graph of $e(\delta, 1)$ and $e(\delta,1.2)$ along with the approximations (dashed lines) in \eqsref{eDeltaLarge}{eDeltaSmall}. 

\begin{figure}[h!] 
\begin{center}
\subfigure[]{\includegraphics[width=.495\textwidth]{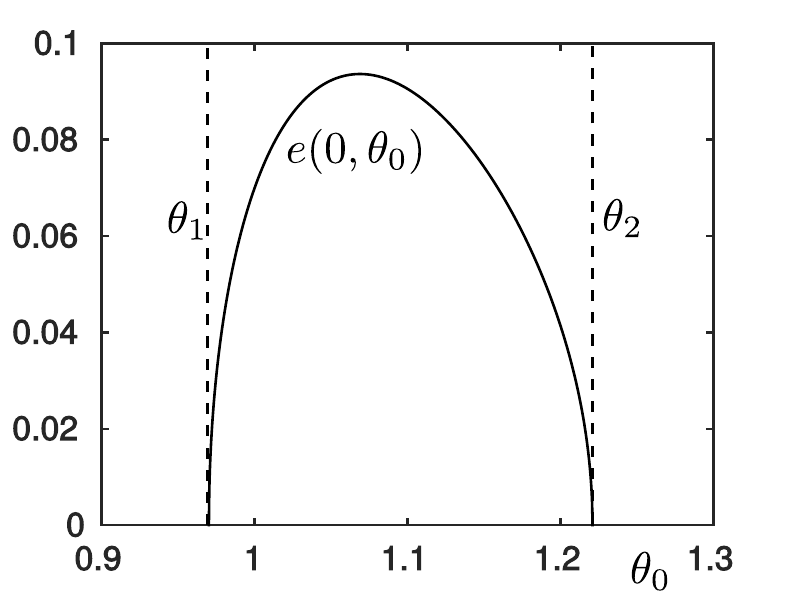}}
\subfigure[]{\includegraphics[width=.495\textwidth]{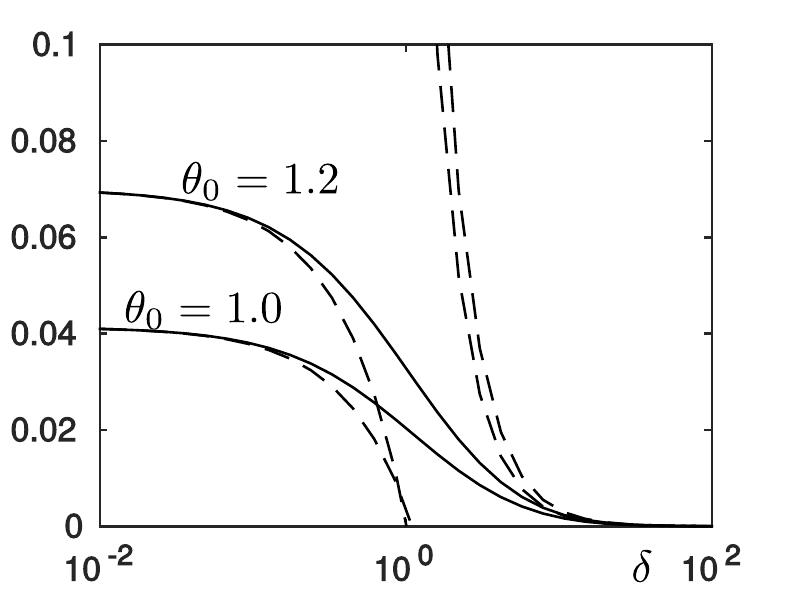}}
\end{center}
 \caption{(a) Graph of $e(0,\theta_0)$ from \eqref{eq:eDeltaZero}, where $\theta_{1,2}$ are given by \eqref{eq:thetacrit}. (b) Graph of $e(\delta,\theta_0)$ for $\theta_0=1$ and $\theta_0=1.2$, where the dashed lines correspond to the approximations obtained from \eqsref{eDeltaLarge}{eDeltaSmall}. For both figures, $\alpha=3$ and $\mu=1.4$. }
 \figlab{fig:eDelta0}
\end{figure}

In the inconsistent case, described by \thmref{thm:main}, the initial conditions \eqref{eq:ics} are very similar to those assumed by \cite{ZhaoLiuChenBrogliato2015}. As in their case, these conditions would be impossible to reach in an experiment without using some form of controller\footnote{To see this, fix any $\hat y_0<0$. Then by applying the approach in section \secref{sec:slipping} backwards in time, it follows that the backward flow of \eqref{eq:ics} for $b<0$ (dashed lines in \figref{g}(a), illustrating the $\kappa_1$-dynamics) intersects the section $\{\hat y=\hat y_0\}$ at a distance which is $o(1)$-close to $\gamma^s\cap \{\hat y=\hat y_0\}$ as $\epsilon\rightarrow 0$. Here $\gamma^s$ is the stable manifold of $C$ for $\epsilon=0$. But cf. \eqref{eq:haty}$_{\epsilon=0}$, the horizontal velocity $v$ (and hence the energy) increases unboundedly along $\gamma^s$ in backwards time. This increase occurs on the fast time scale $\tau$.}. Nevertheless, it should be possible to set up the initial conditions in \eqref{eq:ics} to approach the rigid surface from above, as it appears to have been done in \cite{ZhaoLiuMaChen2008} for the two-link manipulator system.

The indeterminate case described by \thmref{cor} is characterised by an extreme exponential splitting in phase space, due to the stable manifold of $C_1$ in the $\kappa_1$-system \eqref{eqC1}. For example, the blue orbit in \figref{g}(b) lifts off directly with $w=\mathcal O(\epsilon)$. But on the other side of the stable manifold, the green orbit undergoes IWC and then lifts off with $w=\mathcal O(1)$. 
The initial conditions in \thmref{cor} correspond to orbits that are almost grazing ($\dot y=w=\mathcal O(\epsilon),\,\ddot y=\dot w = b>0$) the compliant surface at $y=0$. 
In \figref{video} we illustrate this further by computing the full Filippov system \eqref{eq:Painleve}$_{\epsilon=10^{-3}}$ for two rods (green and blue) initially distant by an amount of $10^{-3}$ above the compliant surface ($y\approx 0.1$, see also $t=0$ in \figref{video}(a)). \figref{video}(a) shows the configuration of the rods at different times $t=0$, $t=0.25$, $t=0.5$ and $t=1$. Up until $t=0.5$, the two rods are indistinguishable. At $t=0.5$, grazing ($\dot y=w\approx 10^{-3}$) with the compliant surface $y=0$ occurs where $\theta\approx 0.9463$, $\phi\approx 1.6654$, $v\approx 1.00$ (so $b\approx 1.2500$ and $p_+\approx -2.243$). The green rod then undergoes IWC, occurring on the fast time scale $\tau$, and therefore subsequently lifts off from $y=0$ with $w=\mathcal O(1)$. In comparison the blue rod lifts off with $w\approx 10^{-3}$. At $t=1$ the two rods are clearly separated. 
\figref{video}(b) shows the projection of the numerical solution in \figref{video}(a) onto the $(w,\hat y)$-plane (compare \figref{fig:IWC}(b)). The blue orbit lifts off directly. The green orbit, being on the other side of the stable manifold of $C_{1}$, follows the unstable manifold (red) until sticking occurs. Then when $\hat F_{N}=0$ at $\hat y+\delta \hat w=0$ (dashed line), lift off occurs almost vertically in the $(w,\hat y)$-plane.  \figref{video}(c) and \figref{video}(d) show the vertical velocity $w$ and horizontal velocity $v$, respectively, for both orbits over the same time interval as \figref{video}(b); note the sharp transition for the green orbit around $t=0.5$, as it undergoes IWC. In \figref{video}(c), we include two dashed lines  $w= ev_0$ and $w=-\frac{p_+}{q_+}v_0$, corresponding to our analytical results \eqref{expr} and \eqref{eq:z2EpsPi2}, which also hold for the indeterminate case (from \thmref{cor}), in excellent agreement with the numerical results.

\begin{figure}[h!] 
\begin{center}
\subfigure[]{\includegraphics[width=.495\textwidth]{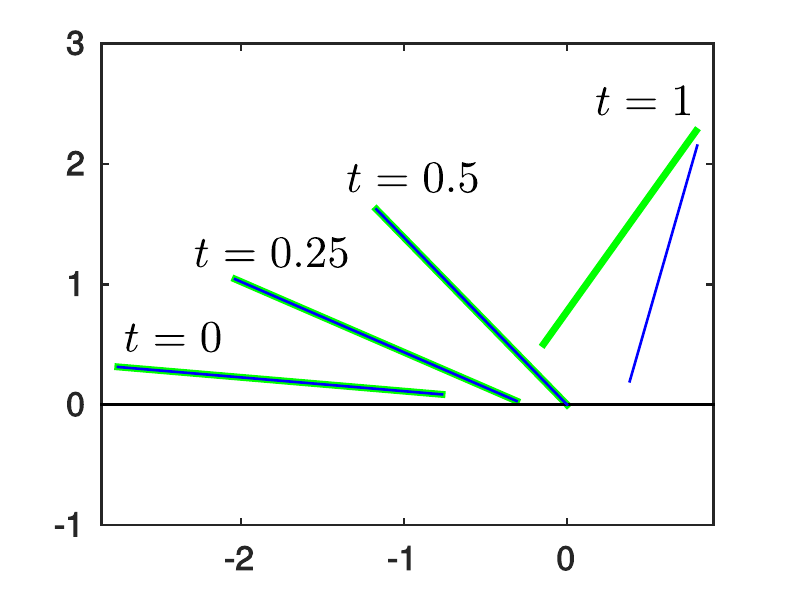}}
\subfigure[]{\includegraphics[width=.495\textwidth]{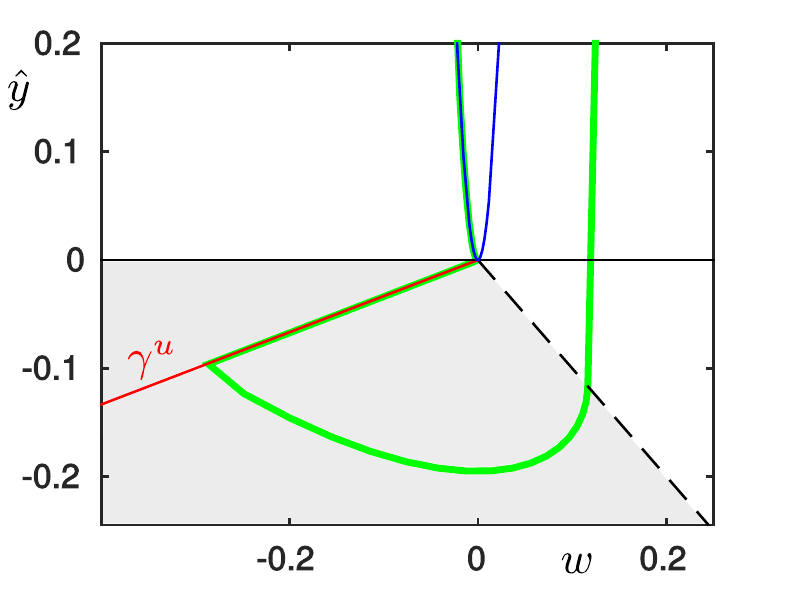}}
\subfigure[]{\includegraphics[width=.495\textwidth]{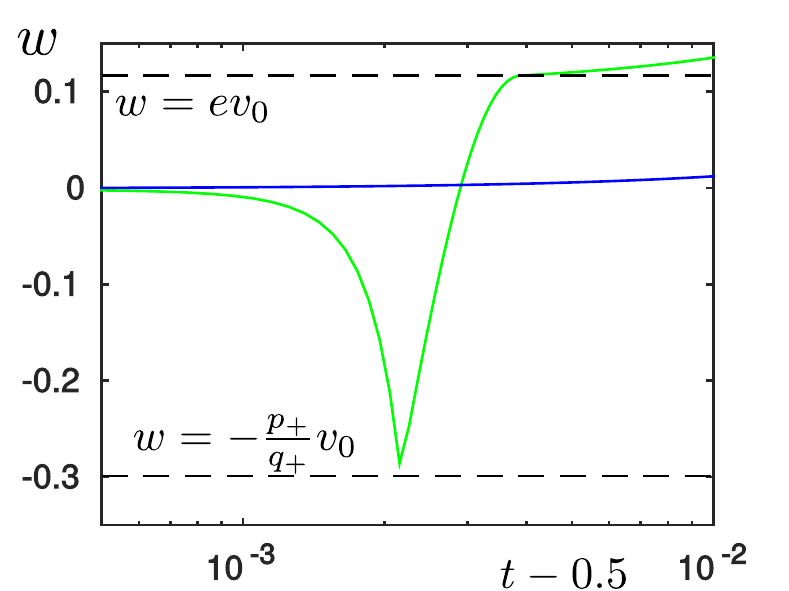}}
\subfigure[]{\includegraphics[width=.495\textwidth]{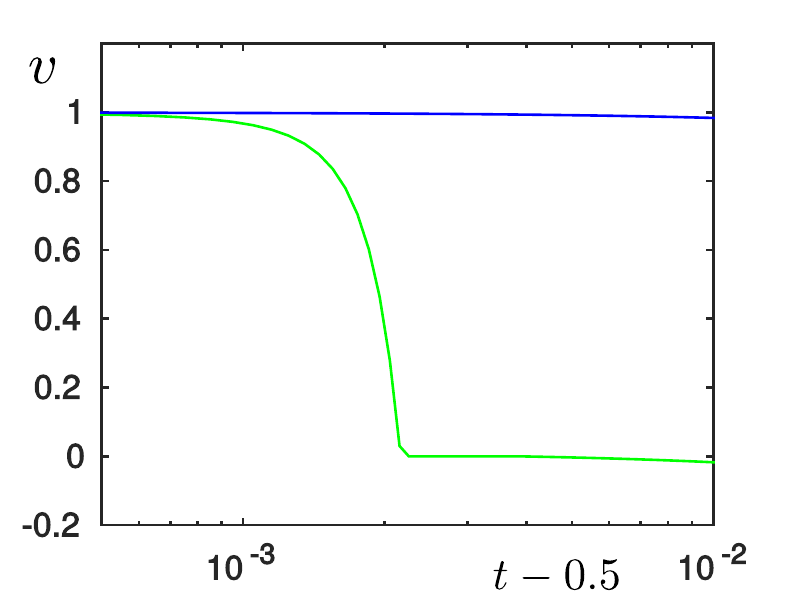}}
\end{center}
 \caption{(a) Dynamics of the Painlev\'e rod described by the Filippov system \eqref{eq:Painleve} for $\mu=\alpha=3$, $\delta=1$ and $\epsilon=10^{-3}$ in the indeterminate case. The green and blue rods are separated at $t=0$ by a distance of $10^{-3}$. At around $t= 0.5$, impact with the compliant surface occurs. The green rod experiences IWC whereas the blue rod lifts off directly. (b) Projection onto the $(w,\hat y)$-plane. (c) and (d) $w$ and $v$ as functions of time near $t=0.5$.}
 \figlab{video}
\end{figure}
 \section{Conclusions}\seclab{conclusions}
We have considered the problem of a rigid body, subject to a unilateral constraint, in the presence of Coulomb friction. Our approach was to regularize the problem by assuming a compliance with stiffness and damping at the point of contact. This leads to a slow-fast system, where the small parameter $\epsilon$ is the inverse of the square root of the stiffness.

Like other authors, we found that the fast time scale dynamics is unstable. Dupont and Yamajako \cite{DupontYamajako1997} established conditions in which these dynamics can be stabilized. In contrast, McClamroch \cite{McClamroch1989} established under what conditions the unstable fast time scale dynamics could be controlled by the slow time scale dynamics. Other authors have used the initial scaling \eqref{eq:initialscaling}, together with the scaling $\kappa_1$ to numerically compute stability boundaries \cite{DupontYamajako1997, McClamroch1989} or phase plane diagrams \cite{ChampneysVarkonyi2016}. 

The main achievement of this paper is to rigorously derive these, and other, results that have eluded others in simpler settings. For example, the work of Zhao {\it et al.} \cite {ZhaoLiuChenBrogliato2015} assumes no damping in the compliance and uses formal methods to provide estimates of the times spent in the three phases of IWC. They suggest that their analysis can  ``$\cdots$ roughly explain why the Painlev\'e paradox can result in [IWC].". In contrast, we assumed that the compliance has {\it both} stiffness and damping, analysed the problem rigorously, derived exact and asymptotic expressions for many important quantities in the problem and showed {\it exactly} how and why the Painlev\'e paradox can result in IWC.  There are no existing results comparable to \eqref{expr}, \eqsref{eDeltaLarge}{eDeltaSmall} for any value of $\delta$.

Our results are presented for arbitrary values of the compliance damping and we are able to give explicit asymptotic expressions in the limiting cases of small and large damping, all for a large class of rigid bodies, including the case of the classical Painlev\'e example in \figref{fig:rod}.

Given a general class of rigid body and a general class of normal reaction, we have been able to derive an explicit connection between the initial horizontal velocity of the body and its lift-off vertical velocity, for arbitrary values of the compliance damping, as a function of the initial orientation of the body.

%
%
%
%

\appendix

\bibliography{refs}

\begin{thebibliography}{10}

\bibitem{Lesuanan1990}
Le~Suan An.
\newblock The {P}ainlev\'e paradoxes and the law of motion of mechanical
  systems with {C}oulomb friction.
\newblock {\em Prikl. Math. Mekh.}, 54:430--438, 1990.

\bibitem{BlumenthalsBrogliatoBertails2016}
A.~Blumenthals, B.~Brogliato, and F.~Bertails-Descoubes.
\newblock The contact problem in {L}agrangian systems subject to bilateral and
  unilateral constraints, with or without sliding {C}oulomb's friction: a
  tutorial.
\newblock {\em Multibody Syst. Dyn.}, 38:43--76, 2016.

\bibitem{Brach1997}
R.M. Brach.
\newblock Impacts coefficients and tangential impacts.
\newblock {\em ASME J. Applied Mechanics}, 64:1014--1016, 1997.

\bibitem{Brogliato1999}
B.~Brogliato.
\newblock {\em Nonsmooth mechanics}.
\newblock Springer, London, 2nd edition, 1999.

\bibitem{ChampneysVarkonyi2016}
A.R. Champneys and P.~V\'arkonyi.
\newblock The {P}ainlev\'e paradox in contact mechanics.
\newblock {\em IMA J. Applied Math.}, 81:538--588, 2016.

\bibitem{Darboux1880}
G.~Darboux.
\newblock \'{E}tude g\'eometrique sur les percussions et le choc des corps.
\newblock {\em Bulletin des Sciences Math\'ematiques et Astronomique, 2e
  serie}, 4:126--160, 1880.

\bibitem{DupontYamajako1997}
P.~E. Dupont and S.~P. Yamajako.
\newblock Stability of frictional contact in constrained rigid-body dynamics.
\newblock {\em IEEE Trans. Robotics Automation}, 13:230--236, 1997.

\bibitem{fen1}
N.~Fenichel.
\newblock Persistence and smoothness of invariant manifolds for flows.
\newblock {\em Indiana University Mathematics Journal}, 21:193--226, 1971.

\bibitem{fen2}
N.~Fenichel.
\newblock Asymptotic stability with rate conditions.
\newblock {\em Indiana University Mathematics Journal}, 23:1109--1137, 1974.

\bibitem{fen3}
N.~Fenichel.
\newblock Geometric singular perturbation theory for ordinary differential
  equations.
\newblock {\em J. Diff. Eq.}, 31:53--98, 1979.

\bibitem{filippov1988differential}
A.F. Filippov.
\newblock {\em Differential Equations with Discontinuous Righthand Sides}.
\newblock Mathematics and its Applications. Kluwer Academic Publishers, 1988.

\bibitem{GenotBrogliato1999}
F.~G\'enot and B.~Brogliato.
\newblock New results on {P}ainlev\'e paradoxes.
\newblock {\em European Journal of Mechanics A/Solids}, 18:653--677, 1999.

\bibitem{Ivanov1986}
A.P. Ivanov.
\newblock On the correctness of the basic problem of dynamics in systems with
  friction.
\newblock {\em Prikl. Math. Mekh.}, 50:547--550, 1986.

\bibitem{jones_1995}
C.K.R.T. Jones.
\newblock {\em Geometric Singular Perturbation Theory, Lecture Notes in
  Mathematics, Dynamical Systems (Montecatini Terme)}.
\newblock Springer, Berlin, 1995.

\bibitem{Keller1986}
J.~B. Keller.
\newblock Impact with friction.
\newblock {\em ASME J. Applied Mechanics}, 53:1--4, 1986.

\bibitem{krihog}
{K. Uldall} Kristiansen and {S. J.} Hogan.
\newblock {On the use of blowup to study regularizations of singularities of
  piecewise smooth dynamical systems in $\mathbb{R}^3$}.
\newblock {\em SIAM Journal on Applied Dynamical Systems}, 14(1):382--422,
  2015.

\bibitem{krihog2}
{K. Uldall} Kristiansen and {S. J.} Hogan.
\newblock Regularizations of two-fold bifurcations in planar piecewise smooth
  systems using blowup.
\newblock {\em SIAM Journal on Applied Dynamical Systems}, 14(4):1731--1786,
  2015.

\bibitem{3rd}
{K. Uldall} Kristiansen and {S. J.} Hogan.
\newblock Le canard de {P}ainlev\'e.
\newblock {\em In preparation}, 2017.

\bibitem{krupa_extending_2001}
M.~Krupa and P.~Szmolyan.
\newblock Extending geometric singular perturbation theory to nonhyperbolic
  points - fold and canard points in two dimensions.
\newblock {\em {SIAM} Journal on Mathematical Analysis}, 33(2):286--314, 2001.

\bibitem{kuehn2015}
C.~Kuehn.
\newblock {\em {Multiple Time Scale Dynamics}}.
\newblock Springer-Verlag, Berlin, 2015.

\bibitem{Lecornu1905}
L.~Lecornu.
\newblock Sur la loi de {C}oulomb.
\newblock {\em Comptes Rendu des S\'eances de l'Academie des Sciences},
  140:847--848, 1905.

\bibitem{LeineBrogliatoNijmeijer2002}
R.~Leine, B.~Brogliato, and H.~Nijmeijer.
\newblock Periodic motion and bifurcations induced by the {P}ainlev\'e paradox.
\newblock {\em European Journal of Mechanics A/Solids}, 21:869--896, 2002.

\bibitem{LiuZhaoChen2007}
C.~Liu, Z.~Zhao, and B.~Chen.
\newblock The bouncing motion appearing in a robotic system with unilateral
  constraint.
\newblock {\em Nonlinear Dynamics}, 49:217--232, 2007.

\bibitem{McClamroch1989}
N.~H. McClamroch.
\newblock A singular perturbation approach to modeling and control of
  manipulators constrained by a stiff environment.
\newblock In {\em Proc. 28th Conf. Decision Contr.}, pages 2407--2411, December
  1989.

\bibitem{NeimarkFufayev1995}
Yu.~I. Neimark and N.~A. Fufayev.
\newblock The {P}ainlev\'e paradoxes and the dynamics of a brake shoe.
\newblock {\em J. Applied Math. Mech.}, 59:343--352, 1995.

\bibitem{NeimarkSmirnova2001}
Yu.~I. Neimark and V.~N. Smirnova.
\newblock Contrast structures, limit dynamics and the {P}ainlev\'e paradox.
\newblock {\em Differential Equations}, 37:1580--1588, 2001.

\bibitem{Or2014}
Y.~Or.
\newblock Painlev\'e's paradox and dynamic jamming in simple models of passive
  dynamic walking.
\newblock {\em Regular and Chaotic Dynamics}, 19:64--80, 2014.

\bibitem{OrRimon2012}
Y.~Or and E.~Rimon.
\newblock Investigation of {P}ainlev\'e's paradox and dynamic jamming during
  mechanism sliding motion.
\newblock {\em Nonlinear Dynamics}, 67:1647--1668, 2012.

\bibitem{Painleve1895}
P.~Painlev\'e.
\newblock Sur les loi du frottement de glissement.
\newblock {\em Comptes Rendu des S\'eances de l'Academie des Sciences},
  121:112--115, 1895.

\bibitem{Painleve1905a}
P.~Painlev\'e.
\newblock Sur les loi du frottement de glissement.
\newblock {\em Comptes Rendu des S\'eances de l'Academie des Sciences},
  141:401--405, 1905.

\bibitem{Painleve1905b}
P.~Painlev\'e.
\newblock Sur les loi du frottement de glissement.
\newblock {\em Comptes Rendu des S\'eances de l'Academie des Sciences},
  141:546--552, 1905.

\bibitem{ShenStronge2011}
Y.~Shen and W.~J. Stronge.
\newblock Painlev\'e's paradox during oblique impact with friction.
\newblock {\em European Journal of Mechanics A/Solids}, 30:457--467, 2011.

\bibitem{SongKrausKumarDupont2001}
P.~Song, P.~Kraus, V.~Kumar, and P.~E. Dupont.
\newblock Analysis of rigid-body dynamic models for simulation of systems with
  frictional contacts.
\newblock {\em ASME J. Applied Mechanics}, 68:118--128, 2001.

\bibitem{Stewart2000}
D.~E. Stewart.
\newblock Rigid-body dynamics with friction and impact.
\newblock {\em {SIAM} Review}, 42:3--39, 2000.

\bibitem{Stronge2015}
W.~J. Stronge.
\newblock Energetically consistent calculations for oblique impact in
  unbalanced systems with friction.
\newblock {\em ASME J. Applied Mechanics}, 82:081003, 2015.

\bibitem{WilmsCohen1981}
E.~V. Wilms and H.~Cohen.
\newblock Planar motion of a rigid body with a friction rotor.
\newblock {\em ASME J. Applied Mechanics}, 48:205--206, 1981.

\bibitem{ZhaoLiuChenBrogliato2015}
Z.~Zhao, C.~Liu, B.~Chen, and B.~Brogliato.
\newblock Asymptotic analysis and {P}ainlev\'e's paradox.
\newblock {\em Multibody Syst. Dyn.}, 35:299--319, 2015.

\bibitem{ZhaoLiuMaChen2008}
Z.~Zhao, C.~Liu, W.~Ma, and B.~Chen.
\newblock Experimental investigation of the {P}ainlev\'e paradox in a robotic
  system.
\newblock {\em ASME J. Applied Mechanics}, 75:041006, 2008.

\end{thebibliography}
\bibliographystyle{plain}
\newpage

 \end{document}